\documentclass[11pt,twoside,a4paper]{amsart}

\usepackage{amsmath, amsfonts, amssymb, microtype}
\usepackage{amsthm,amscd}
\usepackage[english]{babel}
\usepackage{amssymb}
\usepackage{amstext}
\usepackage{latexsym}
\usepackage{amscd}
\usepackage{xcolor}

\begin{document}

\numberwithin{equation}{section}

\newtheorem{thmm}{Theorem} \newtheorem{theorem}{Theorem}[section]
\newtheorem{lemma}[theorem]{Lemma}
\newtheorem{cor}[theorem]{Corollary}
\newtheorem{sublem}[theorem]{Sublemma}
\newtheorem{proposition}[theorem]{Proposition}
\newtheorem{conj}{Conjecture}\renewcommand{\theconj}{}
\newtheorem{defin}[theorem]{Definition}
\newtheorem{ldefin}[theorem]{Lemma-Definition}
\newtheorem{cond}{Condition} \newtheorem{hyp}{Hypotheses}
\newtheorem{exmp}{Example} \newtheorem{remark}[theorem]{Remark}
\newtheorem{convention}[theorem]{Convention}
\newtheorem{definition}[theorem]{Definition}
\newtheorem{sublemma}[theorem]{Sublemma}
\newtheorem{corollary}[theorem]{Corollary}
 \def\III{\mathbb{I}} 
 \def\complex{\mathbb{C}} \def\supp{\mathrm{supp}}
\def\BB{\mathcal{B}}
\def\FF{\mathcal{F}} \def\GG{\mathcal{G}}
\def\II{\mathcal{I}} \def\JJ{\mathcal{J}}
\def\LL{\mathcal{L}}
\def \PP{\mathcal{P}} \def \QQ{\mathcal {Q}}
 \def\TT{\mathcal{T}} 
\def\ZZ{\mathcal{Z}} \def\FFF{\mathbb{F}}
\def\PPP{\mathbb{P}}

\def\real{{\mathbb R}} \def\rational{{\mathbb Q}}
\def\natural{{\mathbb N}} \def\integer{{\mathbb Z}}

\def\B{\mathcal{B}} \def\C{\mathcal{C}} \def\O{\mathcal{C}}
\def\P{\mathcal{P}} \def\U{\mathcal{U}}

\def\const{\operatorname{const}}
\def\closure{\operatorname{closure}}
\def\dist{\operatorname{dist}} \def\esssup{\operatorname{ess\
    sup}} \def\essinf{\operatorname{ess\ inf}}
\def\inte{\operatorname{int}} \def\Lip{\operatorname{Lip}}
\def\max{\operatorname{max}} \def\min{\operatorname{min}}
\def\mod{\operatorname{mod}} \def\osc{\operatorname{osc}}
\def\sign{\operatorname{sign}}
\def\supp{\operatorname{supp}}

\def\al{\alpha} \def\be{\beta} \def\vep{\varepsilon}
\def\th{\theta} \def\om{\omega} \def\la{\lambda} \def\La{\Lambda}
\def\ga{\gamma} \def\ka{\kappa}

\def\ome{\tilde{\om}} \def\a{a_J} \def\I{I} \def\i{[0,\delta]}
\def\ii{[0,\vep]} \def\l{l_0}

\title[Parameter ASIP for the quadratic family] {A parameter ASIP
  for the quadratic family} \author{Magnus
  Aspenberg\textsuperscript{(1)}, Viviane
  Baladi\textsuperscript{(2),(3)}, and Tomas
  Persson\textsuperscript{(1)}} \address{(1) Centre for
  Mathematical Sciences, Lund University, Box 118, 221 00 Lund,
  Sweden}

\address{(2) Institute for Theoretical Studies, ETH, 8092 Z\"urich,
  Switzerland} \address{(3) Sorbonne Universit\'e and Universit\'e Paris Cit\'e, CNRS,   Laboratoire de Probabilit\'es, Statistique et Mod\'elisation,
  F-75005 Paris, France}

\email{magnus.aspenberg@math.lth.se} \email{baladi@lpsm.paris}
\email{tomaspersson@gmx.com} \date{June 4, 2024}

\begin{abstract}
  Consider the quadratic family $T_a(x) = a x (1 - x)$, for
  $x \in [0, 1]$ and mixing Collet--Eckmann (CE) parameters 
  $a \in (2,4)$.  For bounded
  $\varphi$, set
  $\tilde \varphi_{a} := \varphi - \int \varphi \, d\mu_a$, with
  $\mu_a$ the unique acim of $T_a$, and put
  $(\sigma_a (\varphi))^2 := \int \tilde \varphi_{a}^2 \, d\mu_a + 2
  \sum_{i>0} \int \tilde \varphi_{a} (\tilde \varphi_{a} \circ T^i_{a}) \,
  d\mu_a$. For any mixing Misiurewicz parameter $a_*$,
  we find a positive measure set $\Omega_*$ of mixing CE
  parameters, containing $a_*$ as a  Lebesgue density point, 
  such that for  any H\"older $\varphi$
  with $\sigma_{a_*}(\varphi)\ne 0$, there exists $\epsilon_\varphi >0$
  such that, for normalised Lebesgue measure on
  $\Omega_*\cap [a_*-\epsilon_\varphi, a_*+\epsilon_\varphi]$, the functions
  $\xi_i(a)=\tilde \varphi_a(T_a^{i+1}(1/2))/\sigma_a (\varphi)$ satisfy an almost sure
  invariance principle (ASIP) for any error exponent $\gamma
  >2/5$. (In particular, the Birkhoff sums satisfy this ASIP.)
  Our argument goes along the lines of Schnellmann's proof for
  piecewise expanding maps. We need to introduce a variant of Benedicks--Carleson parameter exclusion 
  and to exploit fractional response and uniform exponential decay of correlations
  from \cite{BBS}.
\end{abstract}
\thanks{MA and TP thank the Institute for Theoretical
  Studies, ETH Z\"urich, for its hospitality in 2022.  VB's research is
  supported by the European Research Council (ERC) under the
  European Union's Horizon 2020 research and innovation programme
  (grant agreement No 787304).  VB is grateful to the K.\ and
  A.\ Wallenberg Foundation, for an invitation to Lund
  University in 2021, and to J. Sedro for the reference \cite{Mo}.
  VB and TP thank Q. Berger and F. P\`ene 
  for useful conversations and acknowledge the hospitality of CIRM (Luminy) 
 and the Lorentz Center (Leiden),  where part of
this work was carried out in 2022. We are grateful to the referee for a very careful reading
which allowed us to remove several small mistakes.
}
\maketitle
\tableofcontents

\section{Introduction}

\subsection{Background and Motivation}
Let $(\Omega_*,m_*,\FF_*)$ be a probability spa\-ce.  We say that a
sequence of measurable functions
$\xi_i \colon \Omega_* \to\real$, $i\ge1$ satisfies the
\emph{almost sure invariance principle (ASIP) with error exponent
  $\gamma<1/2$} if there exist a probability space
$(\Omega_W, m_W,\FF_W)$ supporting a (centered) one-dimensional
Brownian motion $W$ and a sequence of measurable functions
$\eta_i \colon \Omega_W\to \real$, $i\ge1$, such that
\renewcommand{\labelenumi}{\roman{enumi})}
\begin{enumerate}
\item The random variables $\{\xi_i\}_{i\ge1}$ and
  $\{\eta_i\}_{i\ge1}$ have the same\footnote{By definition of
    the distribution of discrete-time real-valued stochastic
    processes, this means that for any $n\ge 1$ and any
    $\{y_i\in \real\mid 1\le i \le n\}$, the joint probability
    that $\xi_i\le y_i$ for all $1\le i \le n$ coincides with the
    joint probability that $\eta_i\le y_i$ for all
    $1\le i \le n$.}  distribution.

\item Almost surely, 
  $ \bigl|W(n)-\sum_{i=1}^{n}\eta_i\bigr|=O(n^\gamma) $ as
  $n\to\infty$.
\end{enumerate}

Since a Brownian motion at integer times coincides with a sum of
independent identically distributed (i.i.d.) Gaussian variables,
the above definition can also be formulated as an almost sure
approximation, with error $o(n^{\gamma})$, by a sum of i.i.d.\
Gaussian variables.

It is a classical result (see \cite{ps}) that
if the $\{\xi_i\}$ satisfies the ASIP then it satisfies the law
of the iterated logarithm (LIL), the central limit theorem (CLT)
and the functional CLT: Letting $\sigma^2>0$ be the variance of
the Brownian motion $W$ (the expectation is zero by assumption),
and denoting Lebesgue measure by $m$, the LIL says that
\[
  \limsup_{n\to\infty}\frac1{\sqrt{2n\log\log n}}\sum_{i=1}^n
  \xi_i(a) =\sigma\,,\qquad\text{for $m_*$-almost every }a\in
  \Omega_*\,,
\]
and the CLT  (for the functional CLT, see
  \cite[Lemma~5.1]{DLS})
  says that
\[
  \lim_{n\to\infty}m_* \biggl( \biggl\{\, a \in \Omega_* \mid
  \frac1{\sigma\sqrt{n}} \sum_{i=1}^n\xi_i(a)\le y \,\biggr\}
  \biggr) =\frac1{\sqrt{2\pi}}\int_{-\infty}^ye^{-s^2/2}\, ds\, ,
  \ \forall y \in \real\, .
\]

\bigskip

We consider $I=[0,1]$ and the quadratic family
\[
  T_a(x) = a x (1 - x)\, ,\quad x \in I\, ,\,\, a\in (2,4]\,
  .
\]
Denote by $c=1/2$ the critical point of $T_a$ and set
$ c_{j}(a)=T_a^{j}(c)$ for $j\ge 1$.

If
$\liminf_{n \to \infty} n^{-1} \log \partial_x(T_a^n) (T_a(c)) >
0$, we say that $a$ is a Collet--Eck\-mann (CE) parameter.  If
$a$ is CE, then $T_a$ admits a unique
absolutely continuous invariant probability measure (acim)
$\mu_a=h_a dm$.  
Our
goal is to find a positive Lebesgue measure set
$\Omega_*$ of  CE parameters with a
Lebesgue density point $a_*\in \Omega_*$ such that for
any H\"older continuous function $\varphi \colon I\to \real$ with
$\sigma_{a_*}(\varphi)\ne 0$ (see \eqref{defsig}), there exists $\epsilon_\varphi >0$
such that the ASIP holds for $m_*$ the normalised Lebesgue
measure on
$\Omega_*\cap[a_*-\epsilon_\varphi, a_*+\epsilon_\varphi]$ and
\[
  \xi_j(a):=\varphi_a(c_{j+1}(a))\, , \qquad j\ge 0\, , \quad
  a\in\Omega_*\cap[a_*-\epsilon_\varphi, a_*+\epsilon_\varphi] \,
  ,
\]
where $\varphi_a$ is a suitable normalisation of $\varphi$ (see
\eqref{eq.varphia}). We follow the approach of Schnellmann
\cite{DS}, who developed this program for transversal families of
piecewise expanding maps $T_a$, for which $\Omega_*$ can be taken
to be an interval.
 
Our main motivation is to extend to the quadratic family the
method developed by de Lima--Smania \cite{DLS} in the setting of
piecewise expanding maps, in order to study linear and fractional
response. (This method requires a functional central limit
theorem, see \cite[Lemma~5.1]{DLS}.)

\smallskip

We say that $T_a$ is mixing if it is topologically mixing on
\[
  K(a):=[c_2(a), c_1(a)]\, .
\] 
 It will be convenient below to restrict to mixing maps
$T_a$. Tiozzo recently showed \cite[Cor 3.15]{Tio} (his result
holds in fact for more general unimodal maps) that $T_a$ is
(strongly) mixing for its unique measure of maximal entropy (MME) if its topological entropy is
greater than $\log (2)/2$.
If $a$ is a CE parameter with strongly mixing MME, then  $T_a$ is
topologically mixing on $K(a)$ since the measure of maximal
entropy has\footnote{Indeed, since $T_a$ has no homtervals if
 $a\in CE$, it is
conjugated to its piecewise linear model $F_a$ by a homeomorphism which maps the MME of $F_a$
to the MME of $T_a$,  and the MME of $F_a$ is absolutely continuous
with a positive density on $[F^2_a(c), F_a(c)]$.}  full support there. 
Since the topological entropy of $T_4$
is equal to $\log 2$, and the topological entropy of $T_a$ is
nondecreasing and continuous (in fact H\"older continuous
\cite{Gu}) in $a$, there exists $a_\mathrm{mix}<4$ such that for all
$a \in (a_\mathrm{mix}, 4]\cap CE$, the map $T_a$
is topologically mixing on $K(a)$, and $\mu_a$ is strongly mixing,
with support $K(a)$.

\smallskip
 
Melbourne and Nicol showed \cite{mn} the ASIP in the phase space $x\in K(a)$,
setting $\xi_i=T^i_a(x)$ for a fixed CE map $T_a$, using an
induced uniformly expanding system (then \cite[Section~7]{ps}
provides an ASIP which projects to the ASIP for the original CE
map).  However, to the best of our knowledge, the ASIP in the
parameter $a$ is still open.
 
In the parameter space, typicality (the law of large numbers,
LLN) and the LIL are known: Avila--Moreira
\cite{avilamoreiratypical}, showed that\footnote{Benedicks and
  Carleson established typicality in \cite{BC1} for the Cantor
  set of CE parameters considered there.} for Lebesgue almost
every CE map $T_a$ the critical point is typical for its unique
absolutely continuous invariant measure $\mu_a=h_a dm$:
\begin{equation}
  \label{eq.typical}
  \lim_{n\to\infty} \frac1n \sum_{i=1}^n \varphi (c_i(a)) =
  \int_0^1 \varphi \, d\mu_a \,,\qquad\forall \varphi\in C^0 \,
  .
\end{equation}
For H\"older continuous $\varphi\colon I\to \real$ and a topological mixing CE
parameter $a$, define $\sigma_a( \varphi)\ge 0$ by
\begin{align}\label{defsig}
 ( \sigma_a(\varphi))^2
  &:= \int_0^1 \biggl( \varphi-\int\varphi
    d\mu_a \biggr)^2 \, d\mu_a \\
  \label{eq.sigma}
  &\qquad + 2\sum_{i>0} \int_0^1 \biggl( \varphi - \int\varphi \, d\mu_a
    \biggr) \biggl( \varphi - \int \varphi \, d\mu_a \biggr) \circ
    T^i_a \, d\mu_a\,,
\end{align}
where the sum \eqref{eq.sigma} is finite because topological
mixing (i.e., the fact that the map is nonrenormalisable) implies
\cite{KeNo}
exponential mixing for the acim and H\"older continuous
observables.

In a work in progress, Gao and Shen \cite{GS2} show that,  for Lebesgue almost every  $a$ in 
 the set of  
mixing CE parameters, for every H\"older observable $\varphi$, either $\sigma_a(\varphi)=0$ or the critical point $c$ of $T_a$ satisfies the LIL for $\varphi$, i.e.,
\[
  \limsup_{n\to\infty} \frac1{\sqrt{2n\log\log n}}\sum_{i=1}^n
  \biggl( \varphi(T_a^i(c))-\int\varphi \, d\mu_a \biggr) =
  \sigma_a(\varphi)\,.
\]

\subsection{Statement of the ASIP (Theorem~\ref{t.asip})}

To state our main result, we need more notation
and definitions. 
For $j\ge0$ and $a \in (a_\mathrm{mix},4]$, set
\begin{equation*}
  x_j(a)=c_{j+1}(a)=T^{j+1}_a(c)\, , \qquad
  T_a'(x)=\partial_xT_a(x)\, , \qquad
  x_j'(a)=\partial_a x_j(a)\, .
\end{equation*}
The family $T_a$ is \emph{transversal} at ${a_*}$ if (see
\cite{tsujii}) there exists $C\ge1$ such that
\begin{equation}
  \label{eq.transversal}
  \frac{1}{C}\le \biggl| \frac{x_j'(a_*)}{(T_{a_*}^j)'(c_1(a_*))}
  \biggr| \le C\,,\qquad\forall j\ge1\,.
\end{equation}
By \cite[Theorem 3]{tsujii2},
 all  CE parameters are transversal. We refer to \cite[($NV_t$)]{tsujii} for an
equivalent condition expressed in terms of the postcritical
orbit.

\smallskip

The map $T_a$ is $(H_a, \kappa_a)$-\emph{polynomially recurrent,}
for $\kappa_a\ge 1$ and $H_a\ge 1$, if
\begin{equation}\label{fordistort}
  |x_{j-1} (a) - c| = |T^j_a(c)-c| \ge \frac {1}{j^{\kappa_a}} \, ,
  \qquad \forall j \ge H_a\, .
\end{equation}
  If
$\inf_{j\ge 1} |T_a^j(c)-c|>0$ then $a$ is called a
\emph{Misiurewicz parameter.} 
Misiurewicz parameters are CE and thus transversal.
Avila and Moreira \cite{avilamoreira2003} showed that, for any
$\kappa_0>1$, the set of parameters $a$ which are
$(H_a, \kappa_0)$-polynomially recurrent for some $H_a$ has full
measure in the set of CE parameters. The set of Misiurewicz parameters $a$ is
uncountable (it has full Hausdorff dimension \cite[Thm.~1.4]{Za} but zero
Lebesgue measure). 
\smallskip

Finally, we introduce the normalisation $\varphi_a$: Let
$\varphi$ be bounded such that $\sigma_{a}(\varphi)\ne 0$ for a
mixing CE parameter $a$.  Then the function
\begin{equation}
  \label{eq.varphia}
  \varphi_a(x):=\frac1{\sigma_a(\varphi)}
  \biggl(\varphi(x)-\int_0^1\varphi \, d\mu_a \biggr) \, 
\end{equation}
is well defined and satisfies
\begin{equation}
  \label{eq.varphinormalised}
  \sigma_a(\varphi_a)=1\,\qquad\text{and} \quad \int \varphi_a \,
  d\mu_a=0\, .
\end{equation}

\begin{theorem}[Main Theorem: ASIP]
  \label{t.asip}
  For any  Misiurewicz parameter $a_*\in (a_\mathrm{mix},4)$
  there exists a positive Lebesgue measure set $\Omega_*$ of
  mixing polynomially recurrent parameters, containing $a_*$ as a
  Lebesgue density point, such that for any H\"older continuous
  function $\varphi$ with $\sigma_{a_*}(\varphi)\ne 0$, there
  exists\footnote{The choice of $\epsilon_\varphi$ ensures in
    particular that $\sigma_a(\varphi)\ne 0$ if
    $\sigma_{a_*}(\varphi)\ne 0$.}  $\epsilon_\varphi >0$ such
  that the functions
  \begin{equation}
    \label{eq.xi}
    \xi_n(a) :=\varphi_a (x_{n}(a))= \varphi_a (T_a^{n+1}(c))\,,
    \qquad n \geq 1\, ,
  \end{equation}
  satisfy the ASIP for normalised Lebesgue measure $m_*$ on
  $\Omega_* \cap [a_*-\epsilon_\varphi, a_*+\epsilon_\varphi]$
  and all error exponents $\gamma>2/5$.
\end{theorem}

The value $a_*=4$ is not covered by our arguments for technical
reasons, since $c_1$ and $c_2$ then lie on the boundary of $I$
(see e.g.\ Footnote~\ref{a=4}).
It is possible (but a bit cumbersome) to handle (a one-sided
neighbourhood of) this value by a change of coordinates as in
\cite[Lemma~2.1]{tsujii}.

We expect that the methods\footnote{For example,
    \cite[Lemma~8.1]{A} would replace \cite[Lemma~V.6.5]{DMS} in
    the proof of Proposition~\ref{keyprop}.} 
    of this paper can
  be extended to the case when the ``root'' $a_*$ is 
  mixing, but only Collet--Eckmann and polynomially recurrent (for large enough
  $\kappa_0>1$), 
  instead of Misiurewicz.   We restrict here to
Misiurewicz parameters $a_*$, for the sake of simplicity.
What is most desirable in view of our original motivation to extend
the analysis of \cite{DLS}, is to obtain a ``fatter'' Cantor set $\Omega_*$
(as opposed to a fatter set of root points $a_*$): Indeed, this extension will probably
require the ASIP on a set $\widetilde \Omega$ for which there
exist $\beta>1$ and a full measure subset $\widetilde \Omega_1\subset \widetilde \Omega$  such that
\begin{equation}\label{tscond0}
  \lim_{\epsilon \to 0}
  \frac{m([a-\epsilon, a+ \epsilon] \setminus  \widetilde \Omega)}{\epsilon^\beta} = 0 \, , \,\,
\forall a \in \widetilde \Omega_1
  \, .
\end{equation}
(See \cite[(5), Prop.~F]{BS2}, note that
\cite[Lemma~E]{BS2} even uses $\beta<2$ close to $2$.) Property \eqref{tscond0} is known for all $\beta <2$ for the sets
$\widetilde \Omega_1\subset \widetilde \Omega$ studied\footnote{Beware that Tsujii's result
  cannot be used immediately. In particular, the main argument in
  the construction of the parameter set in Theorem~1 of
  \emph{Pre-threshold fractional susceptibility function:
    holomorphy and response formula,}
arxiv.org/2203.07942,  is flawed.} 
by Tsujii \cite{tsujii}.    For our Cantor set
  $\Omega_*\subset \Omega_{BC}$,  we expect that for any $\kappa>1$, taking $\kappa_0$ large enough in
  Proposition~\ref{keyprop} the factor  $\epsilon^\beta$ in \eqref{tscond0} must be replaced by
  $\epsilon |\log \epsilon|^{-\kappa}$  (see
  \eqref{momo'}), which does not seem good enough. 
Attaining the goal of our original motivation may thus require
establishing the ASIP on a Cantor set having larger density,  and thus   weakening the polynomial lower recurrence in the
  construction (see comments in the next paragraph). We view this as the most desirable improvement of our
main theorem.

\smallskip
To clarify the role of $\Omega_*$, it is useful to compare
Schnellmann's proof with ours. In \cite{DS}, Schnellmann studies suitable
transversal one-parameter families of piecewise expanding interval maps and obtains
a parameter ASIP on a set $\Omega_*$ which is just an interval $[0,\epsilon^\varphi]$ of parameters.
Indeed, existence of an exponentially mixing acim enjoying fractional response
 (with uniform bounds)  holds in an entire interval  $[0,\epsilon^\varphi]$
in his setting
\cite[Prop.~4.3, Lemma4.5]{DS}. So $[0,\epsilon^\varphi]$  is the
baseline parameter space for his analysis. Some parameters in this baseline
cause difficulties (``exceptionally small sets''), but Schnellmann  can get away with just \emph{ignoring} them (taking advantage
of the fact that their total measure is controlled \cite[(III), Theorem~3.2, Lemma~4.1, proof of Lemmas~6.1--6.2]{DS}) instead of \emph{excluding} them
from the baseline.
Our situation is different, since  we need to \emph{exclude}
parameters which do not have an acim or for which exponential mixing
or fractional response (with uniform bounds) does not hold: Our baseline
set is a Cantor set, and the best we can do is to make it as fat as possible.

The polynomial recurrence \eqref{fordistort} in our
  parameter exclusion (Proposition~\ref{keyprop}), which causes the
``thinness'' of $\Omega_*$,
is needed\footnote{See also \eqref{err1} and \eqref{err2}, which may cause a different error exponent.} to apply the results of \cite{BBS}
in Sections~\ref{s.pud} and \ref{regvar} (Propositions~\ref{p.udc} and \ref{p.udc'} on uniform decorrelation
and fractional response, and its consequence, Lemma~\ref{l.varregularity}). 
Due to this we  already \emph{exclude} the parameters which could
have exceptionally small image and we do not need to
\emph{ignore} them (Lemma~\ref{l.goodpartition}, 
compare also the proof of \cite[Lemma~6.1]{DS} with \eqref{111} below).
In addition, we get an easy proof of  the  local distortion estimate 
  \eqref{eq.distortion1}.
If  the required consequences of \cite{BBS} could be 
extended to sets of parameters which  enjoy only exponential recurrence bounds,
    then we could use the (fatter) Benedicks--Carleson Cantor set $\Omega^\varphi_{BC}$ 
as a  baseline instead of $\Omega_*$ (if necessary, the
Benedicks--Carleson technique  could
be replaced by ideas from  Tsujii \cite{tsujii}, Avila--Moreira
\cite{avilamoreira2003}  or Gao--Shen
\cite{GS1}).
Next, one could try to \emph{ignore} 
the parameters with exceptionally small images
in Lemma~\ref{l.goodpartition}.
For \eqref{eq.distortion1}, see also Footnote~\ref{crude}.

\medskip

We also note for the record here that  the characteristic function
$1_{\widetilde \Omega}$ of a fat enough Cantor
set $\widetilde \Omega$ belongs to a Sobolev space
$H^s_q(I)$ with $s>0$ (see \cite[Props 4.9 and 4.10]{HM}).
Thus, working with a Cantor set of larger density may simplify some
of our  arguments
 (in the proof of Proposition~\ref{p.mainestimate}, e.g.).

\medskip

Finally, the results of this paper probably extend to more
general families of smooth unimodal maps. In the present ``proof
of concept'' work, we choose to restrict to the quadratic family.

\subsection{Structure of the Text}

Schnellmann pointed out \cite[p.~370]{DS} that
the ``Markov partitions'' given by the intervals in
the celebrated Benedicks--Carleson \cite{BC1, BC2} parameter exclusion
construction would be the key to extend his
result to nonuniformly expanding interval maps.

Our paper carries out this plan and is organised as follows: After recalling basic facts in
Section~\ref{ss.uniform}, we adapt in Section~\ref{ss.partitions}
the  Benedicks and Carleson 
procedure  to construct, in a neighbourhood of a
 topologically mixing Misiurewicz point $a_*$, a sequence
$\Omega_n\subset \Omega_{n-1}$ where $\Omega_n$ is
a finite union of intervals in $\PP_n$. At each step, some
intervals in $\PP_n$ are partitioned and the intervals which do
not satisfy a time-$n$ {\it polynomial} recurrence assumption are
excluded.  The remaining Cantor set $\Omega_*(a_*)=\cap_n \Omega_n$ is a
positive Lebesgue measure set of parameters satisfying the
Collet--Eckmann property, polynomial returns and 
distortion control, with uniform constants. (Our distortion bound \eqref{eq.distortion1} is new.)  In addition, the
construction ensures that there are no ``exceptionally small''
sets (Lemma~\ref{l.goodpartition}).   Applying
results from \cite{BBS}, this ensures uniform exponential decay
of correlations (Proposition~\ref{p.udc}) and fractional response
(Proposition~\ref{p.udc'}), from which we obtain
regularity of the map $a\mapsto\sigma_a$ (Lemma~\ref{l.varregularity}). 

Sections~\ref{s.switch} and \ref{s.asip} contain the proof of the
ASIP along the lines of \cite{DS}: First approximate the Birkhoff
sum by a sum of blocks of polynomial size
(Sections~\ref{ss.blocks} and \ref{LLN}), then
(Section~\ref{last}) approximate these blocks by a martingale
difference sequence $Y_j$ and apply Skorokhod's representation
theorem linking a martingale with a Brownian motion (see
\cite[Section~3]{ps}).  The usual application
of  the approach of \cite[Chapter~7]{ps} in dynamics uses a
strong independence condition (see \cite[7.1.2]{ps}) which we do
not have (the $\xi_i$'s are not iterations of a fixed map
and there is no\footnote{See the example (see \cite{kac}, p.~646)
  discussed in \cite{DS}.  Also, as pointed out by \cite{DS}, it
  is not clear how to apply the spectral techniques of
  \cite{gouezel} in our setting.} underlying invariant
measure). We replace this strong independence condition by
uniformity of constants in the exponential decay of correlations
(given by \cite{BBS}) which we translate into properties for
the  $\xi_i$ by switching from parameter to phase space (see
Proposition~\ref{p.mainestimate}), giving estimates similar to
those in \cite[Section~3]{ps}.

\bigskip
 
For $\varpi\in (0,1)$, we shall denote by $C^\varpi$ the set of
$\varpi$-H\"older continuous functions
$\varphi \colon I\to \real$, putting
$\|\varphi\|_\varpi=\sup |\varphi|+H_\varpi(\varphi)$, with
$H_\varpi(\varphi)$ the smallest $H_\varpi$ such that
$|\varphi(x)-\varphi(y)|\le H_\varpi |x-y|^\varpi$ for all $x$,
$y$ in $I$. The letter $C$ is used throughout to represent a (large)
uniform constant, which may vary from place to place.

\section{Bounds for the Quadratic Family. The Cantor Set $\Omega_*(a_*)$}

\subsection{Basic Properties}
\label{ss.uniform}

Clearly, the maps 
\[
  a\mapsto T'_a(x)=\partial_xT_a(x)=a(1-2x) \, ,\qquad
  x\mapsto\partial_a T_a(x)=x(1-x)
\]
are Lipschitz continuous uniformly in $x\in I$ and
$a\in (a_\mathrm{mix},4]$, and in addition
\begin{equation}
  \label{eq.la}
  \sup_{x\in I}|T_a'(x)|\le \Lambda:=4\, , \qquad \forall a \in
  (a_\mathrm{mix},4] \, .
\end{equation}  
Each $T_a$ has two monotonicity intervals, with partition points
$0$, $c=1/2$, and $1$.
The following easy lemma replaces\footnote{We do not need 
    as in \cite[(30)]{DS} that $x$ has the same combinatorics
    under $T_{a_1}$ and $T_{a_2}$ up to the $(n-1)$th iteration. We
    thus do not need any analogue of \cite[Sublemma~5.4]{DS}.}
  \cite[(30)]{DS}:
  
\begin{lemma}
  There exists $C<\infty$ such that, for any
  $a_1,a_2\in (2,4]$, we have
  \begin{equation}
    \label{eq.xaa} |T_{a_1}^n(x)-T_{a_2}^n(x)|\le
    C\Lambda^{n}|a_1-a_2|\, ,\qquad \forall x\in I\, , \quad
    \forall n \geq 1 \, .
  \end{equation} 
\end{lemma}

\begin{proof}
  Clearly, $|T_{a_1}(x)-T_{a_2}(x)|\le |a_1-a_2|$.  For $n\ge 2$,
  using the definition \eqref{eq.la} of $\Lambda$, and setting
  $C=\sum_{j=0}^\infty \Lambda^{-j}$, we get
  \begin{align*}
    |T_{a_1}^n(x)&-T_{a_2}^n(x)| \\
                 &\leq
                   |T_{a_1} (T_{a_1}^{n-1} (x)) - T_{a_2}
                   (T_{a_1}^{n-1} (x))|
                   +|T_{a_2} (T_{a_1}^{n-1} (x)) - T_{a_2}
                   (T_{a_2}^{n-1} (x))| \\
                 &\leq |a_1-a_2| + \Lambda |T_{a_1}^{n-1} (x) -
                   T_{a_2}^{n-1} (x)|\\
                 &\leq |a_1-a_2| (1 + \Lambda) + \Lambda^2
                   |T_{a_1}^{n-2} (x) - T_{a_2}^{n-2} (x)| \leq
                   \ldots \\
                 & \leq  |a_1-a_2| \sum_{j=0}^{n-1} \Lambda^j
                   \leq C \Lambda^n |a_1-a_2| \, .
                   \qedhere
  \end{align*}
\end{proof}

\subsection{A Polynomial  Benedicks--Carleson
  Construction ($\Omega_*(a_*)$, $\PP_n$)}
\label{ss.partitions}

For each $j\ge 0$, the function $x_j(a)=T_a^{j+1}(c)$ is a map
from the parameter space $(a_\mathrm{mix},4]$ to the phase space
$I=[0,1]$, with $x_j(a)\in K(a)$ for all $a$.  The transversality
condition \eqref{eq.transversal} says that the derivatives of
$x_j$ and $T_a^j$ are comparable at $a_*$,
so that
statistical properties (such as the ASIP) can be transferred from
the maps $x\mapsto T_a^j(x)$ to the maps $a\mapsto x_j(a)$.
To make this precise, we next construct a sequence of partitions
in the parameter space.  Our starting point is the following
variant of the Benedicks and Carleson Cantor set\footnote{See \eqref{defBC}
for the construction of $\Omega_{BC}$.} $\Omega_{BC}=\Omega_{BC}(a_*)$
(see \cite{BC1, BC2}) associated to a
 Misiurewicz parameter $a_*$ (which is automatically transversal):

\begin{proposition}[The Cantor set
  $\Omega_*=\Omega_*(a_*,\kappa_0)$]
  \label{keyprop}
  Let $a_*\in (a_\mathrm{mix},4]$ be a  Misiurewicz parameter. There
  exist
    $\lambda_{CE}\in (1, \Lambda)$ and $C_0\in (0,1)$ such that,
for any $d_1\in (0,C_0\log \lambda_{CE}/4)$ and $d_0>0$, there
    exists $\epsilon>0$ such that, for any
     $\kappa_0 >1/d_1$,
  for all large enough $N_0\ge 1$ 
there exists a sequence  $\P_j$ of finite sets of pairwise
  disjoint subintervals of
  \[
    \omega_0 := [a_*-\epsilon, a_*+\epsilon] \cap
    (a_\mathrm{mix},4] \, ,
  \]
  such that  $\PP_1=\PP_2=\ldots=\PP_{N_0}$ and, setting 
  \[
    \Omega_{*} = \Omega_*(a_*,\kappa_0):=\bigcap_{j \geq N_0}
    \Omega_j \, , \qquad \text{with} \qquad \Omega_j:=\bigcup_{
      \omega \in \P_j} \omega \, ,
  \]  
we have  $\Omega_{j+1}\subset \Omega_j$ for $j\ge N_0$, and\footnote{The first bound of \eqref{monot} implies
    that $a\mapsto x_j(a)=T_a^{j+1}(c)$ is monotone on
    $\om \in \PP_j$.}
  \begin{align}
    \label{Markov}
    &\forall j \geq 1\, , \quad \forall \omega \in \P_j \, ,
      \quad \forall 0 \leq \ell < j \, , \quad \exists 
      \omega'\in  \P_\ell \mbox{ such that }
      \omega \subset  \omega'\, ,\\
    \label{monot}
    & |x'_j(a)|>0\, ,\quad |T_a^{j+1} (c) - c| > 0 \, , \quad
      \forall a \in \omega \, , \quad \forall \om \in \PP_j\, ,
      \quad \forall j\ge 0\, ,
  \end{align}
  and there exists\footnote{Note that \eqref{l.transversality}
    replaces \cite[Lemma~2.4]{DS}.}
  $C<\infty$ such that, for all $j\geq N_0$ and
  $\om \in \P_j$,
  \begin{align}
    \label{eq:CE}
    &|(T_a^n)'(T_a(c))|  \geq  \lambda_{CE}^n\, ,\qquad
    && \forall N_0 \leq n \leq j\, , \quad \forall a \in \omega\, , \\
    \label{l.transversality}
    &\frac{1}{C}\leq
      \biggl|\frac{x_n'(a)}{(T_a^{n})'(T_a(c))} \biggr| \leq    C\, ,\qquad
    &&\forall N_0  \leq n \leq j\, , \quad \forall a \in \omega\, , \\
    \label{small}
    &|\tilde \omega |\leq C \lambda_{CE}^{-n}
      |x_n (\tilde \omega)|\, , \qquad
    && \forall N_0\leq n \leq j\,
      ,\quad \forall \tilde \omega\subset  \omega\, ,
  \end{align}
  and, moreover,
  \begin{equation}\label{eq:polapproach}
    |T_a^{n+1}(c)-c| > n^{-\kappa_0}\, ,\qquad
    \forall N_0  \leq n \leq j\, , \quad \forall a \in \omega\, .
  \end{equation}
  Finally,
  we have that $a_*\in \Omega_*$ is a Lebesgue density point of\/ $\Omega_*$,
with
  \begin{equation}\label{lem:measureinomega}
    |\Omega_*|\ge (1-d_0 \cdot e_j)|\Omega_{j-1}|\, ,
    \qquad \forall j \geq N_0\, , \quad \text{where} \quad e_j := \sum_{n =
      j}^\infty n^{-d_1 \cdot \kappa_0} \, ,
  \end{equation}
and we  have the more precise (semi-local) bound
   \begin{equation}\label{lem:measureinomega1}
 \sum_{\substack{\omega\in \PP_{\ell}\\ \omega\subset  \omega'}} 
 |\omega \setminus  (\omega\cap \Omega_*)| \le d_0 \cdot e_{\ell-\ell'} | \omega'|\, , \quad
\forall  \omega' \in\PP_{\ell'}\, , 
\,\forall \ell \ge \ell'\ge N_0\, .
\end{equation}
\end{proposition}

See Lemma~\ref{l.goodpartition} below regarding the absence of
exceptionally small sets and Section~\ref{sec.distortion} for a H\"older
distortion property refining \eqref{classic}.

Clearly, \eqref{eq:polapproach} means that any $a\in \Omega_*$ is
$(N_0,\kappa_0)$-polynomially recurrent. 

 The bound
\eqref{lem:measureinomega} implies that the Cantor set $\Omega_*$
has positive Lebesgue measure as soon as $d_1\cdot \kappa_0>1$ (and $N_0$ is large
  enough). 
 Proposition~\ref{keyprop} holds for such 
  $\kappa_0$, but
  we will need the stronger condition
$d_1 \cdot \kappa_0 \geq 11/3$  to use \eqref{lem:measureinomega} 
   in the proof of Proposition~\ref{p.mainestimate}
(and $d_1\cdot \kappa_0>9/5$ for Lemma~\ref{l.LLNy}).

The local bound \eqref{lem:measureinomega1} is used in the proof of Lemma~\ref{l.chiapprox}.

\begin{proof}[Proof of Proposition~\ref{keyprop}]
  Let $r_0 \ge 2$ be a large integer (to be chosen later, with
  $\epsilon\to 0$ as $r_0$ increases).  For $r\ge r_0$, set
  $I_r=I_r^-\cup I_r^+$, where
  \[
    I_r^+ = [ c+e^{-r-1} ,c+ e^{-r} )\,,\,\, I_r^- = ( c-e^{-r} ,
    c-e^{-r-1} ]\,,\,\, U_r=(c-e^{-r},c + e^{-r})\, ,
  \]
  and cover each $I_r^\pm$ by $r^2$ pairwise disjoint intervals
  $I_{r,\ell}^\pm$ of equal size, each $I_{r,\ell}^\pm$
  containing its boundary point closest to $c$.  Let\footnote{The
    constant $\alpha_{BC}$ is usually called $\alpha$, but we
    shall need the letter $\alpha$ for another purpose in
    \eqref{need}.} $\beta_{BC} > \alpha_{BC}>0$ where
  \[
    e^{-n \alpha_{BC}}\le n^{-\kappa_0}\, , \, \, \forall n \ge
    N_0\, ,
  \]
  for $N_0$ a large integer to be chosen later.

  For $a\in (a_\mathrm{mix},4]$, $\nu \ge 1$, and $r\ge r_0$ such
  that $T^\nu_a(c)\in I_r$, the \emph{binding time}
  $p(a)=p(r,a, \nu)$ of $U_r$ with $T^\nu_a(c)$ is the maximal
  $p\in \integer_+ \cup \{\infty\}$ such that
  \[
    |T^{j}_a (x) - T^{j+\nu}_a (c)| \leq e^{-j\beta_{BC} }
    \,,\qquad \forall 1 \leq j\leq p\, ,\quad \forall x\in U_r\,
    .
  \]
  The
  \emph{first free return time} $\nu_1(a)$ of $a\in (a_\mathrm{mix},4]$
  is the smallest integer $j \geq 1$ for which $T_a^j(c) \in U_{r_0}$.
  For an interval $\omega\subset (a_\mathrm{mix},4]$, the
  \emph{first free return time} $\nu_1(\omega)$ is the smallest
  integer $j \geq 1$ for which there exists $a\in \omega$ with
  $T_a^j(c) \in U_{r_0}$.  If there exists  $r=r(\omega)$  such
  that $x_{\nu_1-1}(\omega)\subset I_r$ (recall that
  $T^{\nu_1}_a(c)=x_{\nu_1-1}(a)$), we define the \emph{first
    binding time} of $\omega$ by 
  $p_1(\omega) = \min_{a \in \omega}p(r,a,\nu_1(\omega))$.  For
  $i\ge 2$, define inductively the $i$th \emph{free return time}
  of (suitable) $\omega$ to be the largest integer
  $\nu_{i}(\omega)>\nu_{i-1}(\omega)+p_{i-1}(\omega)+1$ such that
  \[
    T^{j}_a(c)\cap U_{r_0}=\emptyset\, , \qquad \forall
    \nu_{i-1}(\omega)+p_{i-1}(\omega)+1\le j<\nu_{i}(\omega) \,
    ,\,\,\forall a \in \omega\, ,
  \]
  and, for $r(\omega)$ such that
  $x_{\nu_{i-1}-1}(\omega)\subset I_r$, set the $i$th
  \emph{binding time} of $\omega$ to be
  \[
    p_{i}(\omega) = \min_{a \in \omega} p(r,a,\nu_{i-1} (\omega))
    \, .
  \]
  (Similarly, define inductively for $i\ge 2$ and $a$ such that
  $ T^{\nu_{i-1}}_{a}(c)\in I_r$, the pointwise binding times
  $p_i(a)$ and free returns $\nu_i( a)$.)  The iterates between
  $\nu_i(\omega)$ and $\nu_i(\omega)+p_i(\omega)$ form the $i$th
  bound period of $\omega$, those between
  $\nu_{i-1}(\omega)+p_{i-1}(\omega)+1$ and $\nu_{i}(\omega)-1$
  form its $i$th free period. Finally, if there exist
  $a\in \omega$ and $j\ge \nu_1(\omega)$ such that
  $T_a^j(c)\in U_{r_0}$, we say that $j$ is a {\it return time}
  of $\omega$.  (Return times either are free returns
  $\nu_i(\omega)$ or they occur during the bound period.)

  Note that for any fixed $\epsilon$, setting
    $\omega_0=[a_*-\epsilon,a_*+\epsilon]$, there exists
    $N_\epsilon$ such that $x_{N_\epsilon}(\omega_0)$ contains a
    neighbourhood of $c$ (indeed, by transversality, for any
    $a\in \omega_0\setminus\{a_*\}$, there exists $N(a)$ such
    that $T_{a_*}^{N(a)+1}(c)$ and $T_{a}^{N(a)+1}(c)$ lie on
    different sides of $c$).  In particular,
    $\nu_1(\omega_0)<\infty$. Similarly, all $\nu_i(\omega_0)$
    and $p_i(\omega_0)$ are finite.  \smallskip

  Let $W_{a_*}$ be a neighbourhood of $c$ disjoint from
  $\{T^n_{a_*}(c)\mid n\ge 1\}$.  From now on, we
    only consider $r_0$ large enough such that
    $\overline U_{r_0-1}\subset W_{a_*}$.  Set
    $W_{a_*,r_0}^+=W_{a_*}\cap [c+e^{-r_0},1]$ and
    $W_{a_*,r_0}^-=W_{a_*}\cap [0,c-e^{-r_0}]$.  We claim that,
    for any fixed large $r_0$, we have that
    $x_{\nu_1(\omega_0)-1}(\omega_0)$
    contains\footnote{This fact is used before
        \cite[Lemma V.6.8]{DMS}. (There, $W_{a_*}$ is mistakenly
        mentioned instead of $W^\pm_{a_*,r_0}$. Our $r_0$ is
        denoted by $\Delta$ and our $x_n(a)$ is denoted
        $\xi_{n+1}(a)$ in \cite{DMS}.)}  $W_{a_*,r_0}^+$ or
    $W_{a_*,r_0}^-$ for all small enough $\epsilon$. Indeed,
    $x_{\nu_1(\omega_0)-1}(\omega_0)$ is an interval intersecting
    $U_{r_0}$, and $x_{\nu_1(\omega_0)-1}(\omega_0)$ contains
    $T_{a_*}^{\nu_1(\omega)}(c)\notin W_{a_*}$.

  For small $\epsilon>0$ (to be chosen depending on $r_0$), the
  sequence $ \PP_j$ can now be defined\footnote{We refer
    throughout to \cite[Section V.6]{DMS}. The original ideas and
    key estimates appeared previously in the work of Benedicks
    and Carleson \cite{BC1, BC2}.  See Footnote~\ref{orig}.}
  inductively: Start with the single interval
  $ \PP_0= \PP_1=\ldots = \PP_{N_0} =\{\omega_0 \}$, for
  $\epsilon$ small enough such that $\nu_1( \omega_0)\ge N_0$
  (note that $\nu_1(\omega_0)$ increases if $r_0$ increases or
  $\epsilon$ decreases).

  \smallskip

  For $j> N_0$, each $ \omega \in \PP_{j-1}$ is partitioned into
  finitely many (possibly just one) intervals, at least one of
  which will be included into an auxiliary partition $\PP'_{j}$,
  as follows:

  If $j$ is not a free return\footnote{That is, either $j$ is not
    a return, or it is a return within the bound period.}  time
  of $ \omega$, we include $\omega$ in $\PP'_{j}$. If $j$ is a
  free return time of $\omega$ but $x_{j-1}( \omega)$ does not
  contain an interval $I^\pm_{r,\ell}$ (we call this an {\it
    inessential (free) return}), we also include $\omega$ in
  $\PP'_{j}$.

  Otherwise, $j$ is a free return time of $ \omega$ such that
  $x_{j-1}( \omega)$ contains at least one interval
  $I_{r,\ell}^\pm$. We call this an {\it essential (free)
    return.} In that case, we decompose $x_{j-1}( \omega)$ into
  the following intervals:
  \[
    x_{j-1}( \omega)\setminus U_{r_0} \, ,\qquad \{x_{j-1}(
    \omega)\cap I_{r,\ell}^\pm \mid r \geq r_0\, ,\, 1\leq \ell
    \leq r^2\}\, .
  \]
  If $x_{j-1}( \omega)\setminus U_{r_0}\ne \emptyset$, but any of
  the (at most two) connected components of
  $x_{j-1}( \omega)\setminus U_{r_0}$ has size less than
  $e^{-r_0}(1-1/e) r_0^{-2}=|I^\pm_{r_0,\ell}|$, we join it to
  its neighbour
  $x_{j-1}(\omega)\cap I^\pm_{r_0,\ell}=I^\pm_{r_0,\ell}$.  If a
  connected component of $x_{j-1}( \omega)\setminus U_{r_0}$ has
  size larger than $S:=\sqrt {|U_{r_0}|}$, we subdivide it into
  pairwise disjoint intervals of lengths between $S/2$ and $S$.
  If $x_{j-1}(\omega)\cap I_{r,\ell}^\pm\ne \emptyset$, but
  $I_{r,\ell}^\pm$ is not contained in $x_{j-1}(\omega)$ (this
  can happen for at most two intervals $I_{r,\ell}^\pm$), we join
  $x_{j-1}(\omega)\cap I_{r,\ell}^\pm$ to its neighbour
  $x_{j-1}( \omega)\cap I^\pm_{r',\ell'}=I^\pm_{r',\ell'}$.
  Denote by $\{\hat \omega_{r,\ell}\mid r\ge r_0-1\}$ the
  partition of $x_{j-1}(\omega)$ thus obtained, where the index
  $(r,\ell)$ refers to the ``host'' interval $I_{r,\ell}$
  contained in $\hat \omega_{r,\ell}$ if $r\ge r_0$, while
  $\hat \omega_{r_0-1,\ell}\subset I\setminus U_{r_0}$.  Then we
  discard all intervals $\hat \omega_{r,\ell}$ for which
  \begin{equation}\label{polpol}
    e^r\ge (j-1)^{\kappa_0}\, .
  \end{equation}
  Mapping the remaining intervals via the inverse of the
  diffeomorphism (see \cite[Prop.~V.6.2]{DMS}) $x_{j-1}$ gives
  finitely many subintervals of $\omega$ which we include in $\P'_j$.
  Further intervals $\hat \omega_{r,\ell}$ need to be discarded from
  $\PP'_j$, using a requirement denoted $(FA_{j})$ or $(FA'_j)$ in
  \cite[Section V.6]{DMS}, \cite{Mo}, which finally defines
  $\PP_j$. For further use, we denote these remaining intervals by
  \begin{equation}
    \label{omegarl}
    \omega_{r,\ell}=x_{j-1}^{-1}(\hat\omega_{r,\ell})\, .
  \end{equation}

  It is well known\footnote{\label{orig}The original construction in
    \cite{BC1, BC2} is for $a_*=4$, see \cite{Mo} for a
    self-contained account. It extends to 
    Misiurewicz parameters: for CE parameters, the
    condition in \cite[Theorem~6.1]{DMS} is equivalent to
    \eqref{eq.transversal}, taking large enough $k$ in the last
    line of \cite[p.~406, Step~2]{DMS}.}
  \cite{BC1, BC2, DMS} that, if we replace the condition
  \eqref{polpol} (used to discard intervals) by the exponential
  condition
  \begin{equation}\label{discardexpl}
    \omega \cap I_{r,\ell}^\pm\ne \emptyset \quad \mbox{and} \quad
    e^r\ge e^{\alpha_{BC}  (j-1)} \, ,
  \end{equation}
  to construct sequences $\PP_j^{',BC}$ and $\PP_j^{BC}$, then
  there exists $\lambda_{CE}>1$ (called $e^\gamma$ in
  \cite[(V.6.4), Theorem~V.6.2]{DMS}) such that for any small
  enough $\beta_{BC}>\alpha_{BC}>0$ there exist $N_0'$ such that,
  if $r_0$ is large enough and $\epsilon>0$ small enough, then
  the $\PP_j^{BC}$ satisfy \eqref{Markov}--\eqref{small}
  (\eqref{eq:CE} is called $(EX_j)$ in \cite[Section V.6]{DMS})
  for some $C<\infty$, and the following condition
  (noted\footnote{Strictly speaking, the condition $(BA_j)$ does
    not involve the factor $2$, and a condition $(BA'_j)$
    requiring that for each $\omega\in \PP_j^{'BC}$ there exists
    $a\in \omega$ with $|T_a^{n+1}(c)-c| > e^{-n \alpha_{BC}}$
    for $N_0'\le n\le j$ is used in some lemmas. See
    \cite[Section V.6, Step 5]{DMS}.} $(BA_j)$ in the literature)
  holds for all $j\ge N_0'$
  \begin{align}
    \label{eq:expapproach}
    2 |T_a^{n+1}(c)-c| > e^{-n \alpha_{BC}}\, , \qquad
    \forall N'_0  \leq n \leq j\, , \ \forall a \in \omega \,
    \ \forall \omega \in \PP_j^{',BC} \, . 
  \end{align}
  Since $\lambda_{CE}$ does not depend on $\alpha_{BC}$, $N_0$,
  or $N_0'$, we may assume that
  \[
    14\alpha_{BC}<\log \lambda_{CE}
  \]
  and we may replace $N_0$ by $\max\{N_0, N_0'\}$.

  In particular \cite[Prop.~V.6.1, Lemma~V.6.1 b), c)]{DMS} give
  $\gamma_0>0$, $\lambda_{CE}=e^\gamma\in (1, e^{\gamma_0})$, and
  $C_0>0$ (independent of $r_0$ and $\epsilon$) such that, if
  $a\in \Omega_n$ and $\nu_{\ell+1}(a)\le n$, writing $p_\ell$,
  $\nu_\ell$ for $p_\ell(a)$, $\nu_\ell(a)$, we have
  \begin{equation}\label{fromProp2.2}
    \begin{cases}
      |(T_a^{\nu_{\ell+1}-(\nu_\ell +
        p_\ell+1)})'(T^{\nu_\ell+p_\ell+1}_a(c))| \ge C_0
      e^{\gamma_0 (\nu_{\ell+1}-(\nu_\ell + p_\ell))}  \\
      |(T_a^{p_\ell+1})'(T^{\nu_\ell}_a(c))| \ge
      \lambda_{CE}^{p_\ell/4} \, .
    \end{cases}
  \end{equation}
  To establish \eqref{eq:CE} (the bound below will also be used
  for \eqref{choicer0}), one takes $r_0$ such that
  \[
    r_0^2 C_0^2\log \lambda_{CE}   > |\log C_0|\, .
  \]
  The key distortion bound \cite[Prop.~V.6.3]{DMS} gives $C$ such
  that
  \begin{equation} \label{classic}
    \biggl| \frac{x_j'(a_1)}{x_j'(a_2)}
    \biggr| \leq C \, , \qquad \forall  N_0\le j \le n \,
    , \quad
    \forall a_1, a_2\in\omega \, ,
  \end{equation}
  whenever $n+1$ is a free return time of $\omega\in\PP_n$ with
  $x_{n+1}(\omega)\subset U_{r_0/2}$. The bound
  \eqref{l.transversality} follows from \cite[Prop.~V.6.2 and
  Theorem~V.6.2]{DMS}.

Let $\Omega'_j:=\bigcup_{ \om \in \P'_j} \om$, recall
  $\Omega_j$, and define $\Omega_j^{BC}$ and $\Omega_j^{',BC}$
  accordingly, setting 
  \begin{equation}\label{defBC}
    \Omega_{BC}=\Omega_{BC}(a_*,\alpha_{BC})=\cap_j \Omega^{BC}_j
    \,  , \quad \mbox{so that} \quad \Omega_*(a_*)\subset\Omega_{BC}(a_*)\, .
  \end{equation}

  It is easy to check that \eqref{polpol} implies
  \eqref{eq:polapproach} (for returns during a bound period, use
  that $\ell^{-\kappa_0}-e^{-\ell\beta_{BC }}\ge j^{-\kappa_0}$
  for all $N_0\le \ell\le j-1$, up to increasing $N_0$ again).
  Our choice of $N_0$ implies $\Omega_j\subset
  \Omega_j^{BC}$. Also, \eqref{l.transversality} with
  \eqref{monot} imply that all points in $\Omega_*$
are transversal. Since
  \eqref{small} is an immediate consequence of
  \eqref{eq:CE}--\eqref{l.transversality}, it only remains to
  establish that $a_*$ is a Lebesgue density point in $\Omega_*$
  (clearly, $a_*\in \Omega_*$) and that
  \eqref{lem:measureinomega} and \eqref{lem:measureinomega1}
  hold.

  \smallskip

  To show that $a_*$ is a Lebesgue density point of $\Omega_*$,
  we may follow\footnote{We mention a typo there: Although the
    constant $C=C(\epsilon)$ in the unnumbered equation on
    \cite[p.~433]{DMS} tends to zero as
    $\epsilon=|\omega_0|/2\to 0$, the constant $C_0$ is
    (fortunately) uniformly bounded away from zero. See the proof
    of \cite[Lemma~V.6.5]{DMS}.}  \cite[Step~7 of the proof of
  Theorem~V.6.1]{DMS}, replacing $C e^{-i C_0}$ there by
  $C' i^{-\kappa_0}$.

  \smallskip
  We next establish \eqref{lem:measureinomega} and \eqref{lem:measureinomega1}.
  For suitably small $\bar\eta>0$, and for $J_0\ge 1$ such
  that\footnote{Since $\bar\eta$ is independent of $\epsilon$,
    $r_0$, $N_0$, we may take $N_0\ge J_0$.}
  $\prod_{j=J_0}^\infty (1-e^{-\bar \eta j})>3/4$, the parameter
  exclusion rule \eqref{discardexpl} gives $d'_0>0$ (tending to
  zero with $\epsilon$) such that (\cite[Section V.6, Step
  7]{DMS}, \cite[\S 6]{Mo})
  \begin{align}\label{momo}
    \begin{cases}
      |\omega \cap \Omega^{',BC}_j|\ge (1- d'_0 e^{-j \bar
        \eta})|\omega|\,,
      &\forall \omega \in \PP^{BC}_{j-1}\,, \,\, \forall j\ge J_0\, ,\\
      |\Omega^{BC}_j|\ge |\Omega^{',BC}_j|-e^{-j\bar
        \eta}|\omega_0|\, , \,\,& \forall j\ge J_0\, .
    \end{cases}
  \end{align}
  The above implies 
  $|\Omega^{BC}_j|\ge (1- d'_0 e^{-\bar \eta
    j})|\Omega^{BC}_{j-1}|-e^{-\bar \eta j}|\omega_0|$ for
  $j\ge J_0$, and, exploiting that $|\omega_0|=|\Omega^{BC}_n|$
  for all $n\le N_0$ with $N_0\ge J_0$, and using the definition of
  $J_0$, also that
  \[
    |\Omega^{BC}_j| \geq \biggl( \prod_{n=J_0}^j (1- d'_0
    e^{-\bar \eta n}) - \sum_{n=J_0}^j e^{-\tilde \eta n} \biggr )
    |\omega_0| \geq \frac 1 2 |\omega_0| \, , \qquad \forall j
    \geq J_0\, .
  \]
  (By taking larger $J_0$, i.e.\ smaller $\epsilon$, we could
  replace $1/2$ by a number close to $1$.) Thus, applying
  inductively
  \[
    |\Omega^{BC}_j|\geq \bigl ( (1- d'_0 e^{-\bar \eta j})- 2
    e^{-\bar \eta j} \bigr ) |\Omega^{BC}_{j-1}|\, , \qquad
    \forall j\geq J_0 \, ,
  \]
  we find $\tilde \eta>0$ such that for any $j\ge J_0$
  \begin{equation}\label{mutatis}
    |\Omega_{BC}|\geq \prod_{n=j}^\infty (1 - ( d'_0+2)
    e^{-\bar \eta n}) |\Omega_{j-1}^{BC}|
    \geq (1 - (\tilde d_0 + 2) e^{-\tilde \eta j} )
    |\Omega_{j-1}^{BC}|\, .
  \end{equation}

  Recall that we fixed
  $d_1 \in(0, \frac{C_0}{4} \log \lambda_{CE})$ (independently of
  $\kappa_0$). Let $J_1$ be such that
  $\prod_{j=J_1}^\infty (1-e^{-\bar \eta j}-j^{-2})>3/4$ and
  return to the sets $\Omega_j$, $\Omega_j'$ constructed using
  the (polynomial) exclusion rule \eqref{polpol} for
  $\kappa_0 >1/d_1$.  We claim that for any $d_0>0$, if
  $\epsilon$ is small enough,
  \begin{align}\label{momo'}
    \begin{cases}
      |\omega \cap \Omega'_j|\ge (1-d_0 \cdot j^{-d_1 \kappa_0})
      |\omega|\,,
      &\forall \omega \in \PP_{j-1}\,,\,\,  \forall j\ge J_1\, ,\\
      | \Omega_j|\ge |\Omega'_j|- e^{-j\bar \eta}|\omega_0|\, ,
      &\forall j\ge J_1\, .
    \end{cases}
  \end{align} 
  Before establishing this claim, we note that, mutatis mutandis,
  \eqref{momo'} combined with the arguments leading to
  \eqref{mutatis} implies \eqref{lem:measureinomega}, while the
  more precise claim \eqref{lem:measureinomega1} follows from the
  refinement of \eqref{momo'} coming from the second statement of
  \cite[Lemma~V.6.9]{DMS} (see the use of \cite[Lemma~6.3]{Mo} in
  \cite[Lemma~6.4--Prop.~6.5]{Mo}).  \smallskip

  To show \eqref{momo'}, we proceed in three steps, performing
  the necessary changes in the proof in \cite[Section
  V.6]{DMS}. Recall \eqref{omegarl}.

  Firstly, up to taking larger $N_0$, the conclusion of
  \cite[Lemma~V.6.5]{DMS} (which deals with $(BA'_{j})$ for
  $\omega \in \PP_{j-1}$ satisfying $(BA'_{j-1})$ and
  $(EX_{j-1})$ and having a return at time $j$), if we replace
  the exponential rate $(BA'_{j-1})$ there by our polynomial rate
  \eqref{eq:polapproach}, becomes
  \begin{equation}
    \frac{|\omega \setminus \bigcup_{r \geq \kappa_0\log j}\,
      \omega_{r,\ell}|}{|\omega|}\ge 1- C j ^{-d_1 \kappa_0}\,
    ,\qquad \forall j \geq N_0 \, .
  \end{equation}
  To show this first claim, use that the constant $C_0 \in (0,1)$
  (introduced above) is independent of $\kappa_0$ (because
  $\lambda_{CE}$ does not depend on $\kappa_0$), and that
  \cite[Lemma~V.6.1]{DMS} gives that the bound period $p$ of a
  free return $\nu<j$ with
  \begin{equation}\label{forp}
    I_{r',\ell'}\subset x_\nu(\omega)\, , \quad \mbox{for} \quad
    r_0\leq r' \leq \kappa_0\log \nu\leq \kappa_0 \log j\, ,
  \end{equation}
  satisfies $p\ge C_0 r'$.  Then, up to taking larger $N_0$, we
  can replace \cite[V.(6.20)]{DMS} in the proof of
  \cite[Lemma~V.6.5]{DMS} by
  \begin{equation}\label{6.20}
    |x_j(\omega)|\geq \lambda_{CE}^{ p/4} \frac {e^{-r'}}{(r')^2}
    \geq  \frac {e^{(-1+ d_1)r'}}{(r')^2}
    \geq  \frac {1 }{j^{\kappa_0(1-d_1)}}
    \, , \qquad j \geq N_0\, ,
  \end{equation}
  where we used $d_1 \leq \frac{C_0}{4} \log \lambda_{CE}$ in the
  second inequality. We can thus replace the chain of
  inequalities after \cite[V.(6.20)]{DMS} (using the distortion
  bound \eqref{classic} for $\tilde \omega\subset \omega$ the
  largest interval with $x_n(\tilde \omega)\subset U_{r_0/2}$,
  taking $\epsilon$ small enough and $N_0$ large enough such that
  \eqref{6.20} also holds for $\tilde \omega$) by
  \[
    \frac{|\bigcup_{r \ge \kappa_0\log j}\,
      \omega_{r,\ell}|}{|\omega|} \le \frac{|\bigcup_{r \ge
        \kappa_0\log j}\, \omega_{r,\ell}|}{|\tilde \omega|} \le
    C\frac 1{j^{\kappa_0}}\frac{1}{|x_j(\tilde \omega)|} \le C
    j^{-d_1 \cdot \kappa_0}\, .
  \]

  Secondly,\footnote{We mention here a typo: \cite[V.(6.24)]{DMS}
  follows from \cite[V.(6.22)]{DMS} (and not \cite[V.(6.20)]{DMS}
  as stated there).}  \cite[Lemma~V.6.6]{DMS} (which deals with
  $(FA_j)$) uses \eqref{eq:expapproach} only via
  \cite[Lemma~V.6.3]{DMS}, while \cite[Lemma~V.6.3]{DMS} still
  holds (with the same proof) if we replace
  \eqref{eq:expapproach} by our stronger assumption
  \eqref{eq:polapproach}.

  Thirdly, \cite[Lemmas~V.6.7--6.9]{DMS} are unchanged, 
  establishing \eqref{momo'}.
\end{proof}

Lemma~\ref{l.goodpartition} below is the analogue of \cite[(III)']{DS}): 
\begin{lemma}[No Exceptionally Small Sets]
  \label{l.goodpartition}
  For any $\kappa_1 >\kappa_0$ there exists $N_1\ge N_0$ such that
  $|x_j(\om)|> j^{-\kappa_1}$ for all $j\ge N_1$ and
  $\om\in\PP_j=\PP_j(a_*,\kappa_0)$.
\end{lemma}

\begin{proof}
  We first show the lemma assuming that there exists
  $d_2\in (0,1)$ such that for any $j\ge N_0$, and any
  $ \omega\in \PP_{j}$, we have  \begin{equation}
    \label{worse}
    |x_j(\omega)| \geq \frac{d_2 e^{-r_0}(1-1/e)}{( \kappa_0\log
      j)^2 j^{\kappa_0}}
    \, ,
  \end{equation}
  with $r_0$ as in the proof of
  Proposition~\ref{keyprop}. Indeed \eqref{worse} implies that 
      \[
      |x_{j}(\omega)|\geq \frac{d_2 e^{-r_0}(1 -
        e^{-1})}{\kappa_0^2} \frac{1}{j^{\kappa_0} (\log j)^2} \,
      ,\qquad \forall \omega\in\PP_j\, , \forall j\ge N_0\, .
    \]
    Clearly, there exists $N_1(\kappa_1)\ge N_0$ such that
    the right-hand side is larger than $j^{-\kappa_1}$ for all
    $j\ge N_1$.

  To establish \eqref{worse}, we shall use
  \eqref{eq:polapproach}.  If $j+1$ is an essential free return
  time of $\omega$, then taking $r$ minimal such that
  $x_{j}(\omega)$ contains an interval $I_{r,\ell}^\pm$,
  \begin{equation}\label{casezero}
    |x_{j}(\omega)|\ge |I_{r,\ell}^\pm|=e^{-r}\frac {1-1/e}{r^2}> 
    \frac{j^{-\kappa_0}(1-1/e)}{( \kappa_0\log j)^2} \, . 
  \end{equation}

  Otherwise, letting $j'+1=\nu_{i'}(\omega)\ge \nu_1(\omega)$ be the
  largest essential free return time of $\omega$ such that
  $ j'+1<j+1$, we have $\omega \in \P_{j'}$ (since if
    $\tilde{\omega} \supset \omega$,
    $\tilde{\omega} \in\mathcal{P}_{j'}$, then $\tilde{\omega}$ is
    never cut between time $j'$ and $j$), so
  that \eqref{casezero} implies
  \[
    |x_{j'}( \omega)|> \frac{1-1/e}{( \kappa_0\log
      j')^2(j')^{\kappa_0}}>\frac{1-1/e}{( \kappa_0\log
      j)^2j^{\kappa_0}}\, .
  \]
  We shall combine the above bound with \cite[Lemma~V.6.3,
    Props~V.6.1--6.2]{DMS} to handle the three cases left,
  namely: the time $j+1$ is an inessential free return of $\omega$,
  the time $j+1$ is a return within a bound period of $\omega$, and
  the intersection of $x_j(\omega)$ and $U_{r_0}$ is empty.

  If $j+1=\nu_i( \omega)$ is an inessential free return then
  \cite[V.(6.15) in Lemma~V.6.3]{DMS} gives, for $i'\le i$ as defined above,
  \begin{equation}\label{caseone}
    |x_{j}(\omega)|\ge 2^{i-i'} |x_{j'}(\omega)|>
    2^{i-i'} \frac{1-1/e}{( \kappa_0\log j)^2j^{\kappa_0}}
    \, .
  \end{equation}

  If $j+1$ is a return within the bound period of a previous free
  return $j''+1$ of $\omega$, then using \eqref{casezero} for the
  bound period of an essential return, respectively
  \eqref{caseone} for the bound period of a nonessential return,
  and applying the first claim of \cite[Lemma~V.6.3]{DMS}, we
  find $d_2\in (0,1)$ such that
  \begin{equation}\label{casetwo}
    |x_{j}(\omega)|\ge d_2 \lambda_{CE}^{j-j''} |x_{j''}(\omega)|>
    \frac{d_2(1-1/e)}{( \kappa_0\log j)^2j^{\kappa_0}}
    \, .
  \end{equation}
  If  $x_j(\omega)\cap U_{r_0}=\emptyset$ 
  then \cite[V.(6.2) in Prop.~V.6.1 and
  Prop.~V.6.2]{DMS}
and \eqref{casezero} give
  \begin{equation}\label{casethree}
    |x_{j}( \omega)|\ge d_2 e^{-r_0}  |x_{j'}( \omega)|>
    \frac{d_2 e^{-r_0}(1-1/e)}{( \kappa_0\log j)^2j^{\kappa_0}}
    \, .
  \end{equation}
  We have shown \eqref{worse} and thus
  Lemma~\ref{l.goodpartition}.
\end{proof}

\subsection{A H\"older Local Distortion Estimate}
\label{sec.distortion}

From now on, let $a_*\in (a_\mathrm{mix}, 4)$ be a 
Misiurewicz parameter, fix
$\kappa_0\ge 11 / 3 d_1$, and let
$\Omega_*=\Omega_{*}(a_*,\kappa_0)\subset \Omega_{BC}=
\Omega_{BC}(a_*)$ be the positive measure Cantor set constructed
in Section~\ref{ss.partitions} via families
$\PP_j=\PP_j(a_*,\kappa_0)$.  The following\footnote{For
  \cite[(30)]{DS}, see \eqref{eq.xaa}.  We do not need
  \cite[(32)]{DS}.} replaces \cite[(33), (31)]{DS}. The bound
\eqref{eq.distortion1} is new.

\begin{lemma}[H\"older Distortion Bounds]
  \label{l.distortion}
  There exists $C<\infty$ such that for all
  $n\ge N_0$ (with $N_0$ as in
  Proposition~\ref{keyprop}) and any
  $\om\in \PP_n=\PP_n(a_*,\kappa_0)$ 
   \begin{equation}
     \label{eq.distortion3}
     \frac 1 C \leq
     \biggl|\frac{x_n'(a)/x_j'(a)}{(T_a^{n-j})'(x_j(a))} \biggr| \leq C
     \, , \qquad \forall 1 \leq j \leq n \, ,\quad \forall a \in
     \omega \,.
  \end{equation}
  In addition, there exist $C<\infty$ and $M_0>\kappa_0$
such that, for all
  $n\ge N_0$, each
  $\tilde \omega \in\PP_n=\PP_n(a_*,\kappa_0)$, and every
  $\om\subset \tilde \omega$ and
    $\alpha \in [0,1)$ satisfying 
  \begin{equation}
    \label{need}
    |x_{n}(\omega)| \leq n^{-M_0/(1-\alpha)} \, ,
  \end{equation}
  we have
  \begin{equation}
    \label{eq.distortion1}
    \biggl| \frac{x_n'(a_1)}{x_n'(a_2)}
    \biggr| \leq 1 +
    C|x_{n}([a_1,a_2])|^\alpha \, , \qquad
    \forall a_1, a_2\in\omega \,.
  \end{equation} 
\end{lemma}

If $\alpha=0$, and $n+1$ is a free return of $\omega \in \PP_n$,
the bound \eqref{eq.distortion1} is just \eqref{classic}. We
shall require \eqref{eq.distortion1} for some $\alpha >0$ in Corollary~\ref{c.switch}.

\begin{proof}
  The bound \eqref{eq.distortion3} is an immediate consequence of
  \eqref{l.transversality}.
  
We first claim that\footnote{\label{crude}Our  proof is inspired from that of \cite[Theorem~V.6.2]{DMS}. This is suboptimal but  enough for our purposes. Adapting instead \cite[Lemma~V.6.4]{DMS} could enhance \eqref{eq.distortion1}.} there exist $C'$ and $\kappa_2>0$
such that for any $n$ 
\begin{equation}\label{need2}
\sum_{i=0}^{j-1} |x_i(\om)|\le  C'  j^{\kappa_0+1+\kappa_2 } |x_{j}(\om)| \, , \quad \forall 1\le j\le n \, ,\,\forall \om\subset \tilde \omega \in \P_n\, .
\end{equation}
To start, there is $C$ such that for any $0\le i\le j\le n$, using  \eqref{Markov}  and \eqref{eq.distortion3} there exists $a=a(i,j,\omega)\in \omega$ such that,
setting   $X_{i,j}= x_j \circ x_i^{-1}$,
  \begin{align} \label{addquote}
    \frac{|x_i(\omega)|}{ |x_{j}(\omega)| }
    &= \frac{|x_i(\omega)|}{   |X_{i,j} (x_{i}(\omega))|}= 
    \frac{|x'_i(a)|}{   |x'_{j}(a)|}
   \le    \frac{C}{ |(T^{j-i}_a)'(T^{i+1}_a(c))| }
    \,   .
  \end{align}
(We used  $X'_{i,j}=x'_j/x'_i$ and the mean value theorem in the second equality.)
Next, let  $s_j(a)$  be the largest $\ell$ with $\nu_{\ell}(a)\le j$,
and put\footnote{The condition $(FA)_n$ implicitly used  in Proposition~\ref{keyprop}   says that,
for some fixed arbitrarily small $\tau>0$,
$
F_\ell(a)\ge \ell (1-\tau)$ for $N_0\le \ell \le n$. We shall not need this here.}
\begin{align*}
&q_\ell(a)=\nu_{\ell+1}(a)-(\nu_\ell(a)+p_\ell(a)+1)\,, \qquad \ell=0, \ldots, s_j(a)-1\, , \\
&q_{s_j(a)}(a)=\max\{0,
j-(\nu_{s_j(a)}(a)+p_{s_j(a)}(a)+1)\} \, , \quad F_j(a)= \sum_{\ell=0}^{s_j(a)} q_\ell(a)  .
\end{align*}
Set $p_\ell=p_\ell(a)$, $\nu_\ell=\nu_\ell(a)$, $q_\ell=q_\ell(a)$,
and $s_j=s_j(a)$.  Assume first that $i=0$. Then, we have   (see e.g.\ \cite[V.(6.11)]{DMS})
\begin{align}
\nonumber|(T^{j-i}_a)'(T^{i+1}_a(c))|&=|(T^{j}_a)'(T_a(c))|\\
\nonumber&=|(T^{\nu_{1}-1}_a)'(T_a(c))|
 \cdot
|(T^{j+1-\nu_{s_j}}_a)'(T^{\nu_s}_a(c))|\\
\nonumber&\qquad \cdot \biggl (\prod_{\ell=1}^{s_j-1}|(T^{p_{\ell}+1}_a)'(T^{\nu_\ell}_a(c))|
|(T^{q_\ell}_a)'(T^{\nu_\ell+p_\ell+1}_a(c))| \biggr )\, .
\end{align}
Since $a$ satisfies $(BA)_m$ and $(FA)_m$
for all $m\le n$, the bounds \eqref{fromProp2.2} give $\lambda_{CE}$, $\gamma_0>0$, and $C_0>0$ such that
\[
\prod_{\ell=1}^{s_j-1}|(T^{p_{\ell}+1}_a)'(T^{\nu_\ell}_a(c))|
|(T^{q_\ell}_a)'(T^{\nu_\ell+p_\ell+1}_a(c))| 
\ge \prod_{\ell=1}^{s_j-1} C_0 e^{\gamma_0 q_\ell} \lambda_{CE}^{p_\ell/4} \, .
\]
Similarly, $|(T^{\nu_{1}-1}_a)'(T_a(c))|>C_0 e^{\gamma_0 \nu_1}$. Next,  
if $j \le  \nu_{s_j}+p_{s_j}+1$, we have
\[
|(T^{j+1-\nu_{s_j}}_a)'(T^{\nu_{s_j}}_a(c))|\ge C_0^2 \lambda_{BC}^{j-\nu_{s_j}} j^{-\kappa_0} e^{-r_0}\, , 
\]
where we
used \eqref{eq:polapproach} and \cite[Lemma~V.6.1.b, Prop.~V.6.1]{DMS}.
If $j >  \nu_{s_j}+p_{s_j}+1$, we have, using \cite[Lemma~V.6.1.c, Prop.~V.6.1]{DMS}
\[
|(T^{j+1-\nu_{s_j}}_a)'(T^{\nu_s}_a(c))|\ge C_0^2 \lambda_{BC}^{p_{s_j}/4} 
e^{\gamma_0(j-(\nu_{s_j}-p_{s_j}-1))} e^{-r_0} \, , 
\]
Summarising, 
\[
|(T^{j}_a)'(T_a(c))|\ge  \frac{C_0^{s_j+4}}{C}
\lambda_{CE}^{(j-F_j(a))/4} e^{\gamma_0 F_j(a)}j^{-\kappa_0}\, .
\]
Since $p_\ell\ge C_0 r_0$  (see after \eqref{forp}),
we have $j-F_j\ge j C_0 r_0$ while $s_j\le j/(C_0r_0)$.
We  took $r_0$ large enough (see after \eqref{fromProp2.2}) such that 
\begin{equation}\label{choicer0}
C_0^{s_j+4}
\lambda_{CE}^{(j-F_j(a))/4}\ge 1\, . 
\end{equation}
Finally, using the trivial bound $e^{\gamma_0 F_{j}(a)}\ge 1$, we find
\[
|(T^{j}_a)'(T_a(c))|\ge  \frac{ j^{-\kappa_0}}{C}\, .
\]

If $i\ge 1$ and $\nu_{\ell_i}(a)+p_{\ell_i}(a)<i<\nu_{\ell_i+1}(a)$
for some $\ell_i\ge 1$,
then we  proceed as for $i=0$, replacing $|(T^{\nu_{1}-1}_a)'(T_a(c))|$
by $|(T^{\nu_{\ell_i+1}-i}_a)'(T^{i+1}_a(c))|$, and setting
$F_{i,j}(a)=\nu_{\ell_i+1}(a)-i+\sum_{\ell\ge \ell_i+1}^{s_j(a)} q_\ell(a)$.
Then
\[
|(T^{j-i}_a)'(T^{i+1}_a(c))|\ge \frac{ C_0^{s_j-s_i +4}}{C}
\lambda_{CE}^{((j-i)-F_{i,j}(a))/4} e^{\gamma_0 F_{i,j}(a)}j^{-\kappa_0}\, .
\]
We have $j-i-F_{i,j}\ge (j-i) C_0 r_0$ while $s_j-s_i\le (j-i)/(C_0r_0)$,
and  we find,
using  $e^{\gamma_0 F_{i,j}(a)}\ge 1$ (we do not
know or need  $F_{i,j}(a)\ge (1-\tau)(j-i)$),
\[
|(T^{j-i}_a)'(T^{i+1}_a(c))|\ge\frac{  j^{-\kappa_0}}{C}\, .
\]

Otherwise, $\nu_{\ell_i}(a)\le i-1\le \nu_{\ell_i}(a)+p_{\ell_i}(a)$
for some $\ell_i\ge 1$. There may be (nonfree) returns
during the $\ell_i$th bound period. To bypass this difficulty, we exploit
that
 the length of the $\ell$th bound period
is of the order $r$ if $x_{\nu_{\ell}}(a)\in I_{r}$ (\cite[Lemma~V.6.1a]{DMS}).
By \eqref{polpol}, we have $r_{\ell_i}=O(\log(\nu_{\ell_i}))\le  C \kappa_0 \log i$. Thus, the missing
factor
in the $\ell$th bound period is $\le \Lambda^{ C\kappa_0\log i}\le i^{\kappa_2}$,
and 
\[
|(T^{j-i}_a)'(T^{i+1}_a(c))|\ge  \frac{ j^{-\kappa_0}}{ C i ^{\kappa_2 }}\, .
\]
Summing over $i$,  and recalling \eqref{addquote}, this establishes \eqref{need2}.

\medskip

Next, taking $a_1,a_2 \in \omega$, note that \eqref{small}
  (using the first bound of \eqref{monot} if $i<N_0$) implies that
  for all $N_0\le i\le n$, recalling $T'_a(x)=a(1-2x)$,
\begin{align}
    \nonumber
    |T_{a_1}'(x_i(a_1))&-T_{a_2}'(x_i(a_2))|\\
    \nonumber
                       & \leq |T_{a_1}'(x_i(a_1)) - T_{a_2}' (x_i
                         (a_1))| + |T_{a_2}' (x_i (a_1)) -
                         T_{a_2}' (x_i (a_2))| \\
    \label{eq.36}
                       &\leq |a_1-a_2| + 2 a_2 |x_i(a_1) -
                         x_i(a_2)| \leq (C + 2 a_2) |x_i(\omega)|
                         \,.
\end{align}
  (Note that \eqref{eq.36} replaces \cite[(36)]{DS}.) We claim
  that there exists $C''$ with
  \begin{equation}
    \label{eq.37}
    \biggl|\frac {(T_{a_1}^j)'(x_0(a_1))}{(T_{a_2}^j)'(x_0(a_2))}\biggr|
    \leq 
    1 + C'' e^{C''}  j^{2\kappa_0+1+\kappa_2} |x_{j}(\omega)| \, , \,
    \forall 1\leq j\leq n \, .
  \end{equation}
  (The above replaces \cite[(37)]{DS}.)  Indeed, using the
  classical bound
  \[
    \prod_{i=0}^{j-1} (1+ \upsilon_i) \leq \exp \biggl(
    \sum_{i=0}^{j-1}\upsilon_i \biggr) \leq 1+
    e^{\sum_{i=0}^{j-1}\upsilon_i}\sum_{i=0}^{j-1}\upsilon_i\,,\,\,
    \mbox{ if all }\upsilon_i\geq 0\, ,
  \]
  we have, setting $C''=2C'C(C+2a_2 )$,
  \begin{align}
    \nonumber \biggl| \frac{(T_{a_1}^j)' (x_0
    (a_1))}{(T_{a_2}^j)' (x_0 (a_2))}
    & \biggr| = \biggl| \prod_{i=0}^{j-1} \frac{T_{a_1}' (x_i
      (a_1))}{T_{a_2}' (x_i (a_2))} \biggr| \\
    \nonumber
    &\leq 1 + e^{\sum_i \Bigl|-1 + \frac{T_{a_1}' (x_i
      (a_1))}{T_{a_2}' (x_i (a_2))} \Bigr|} 
      \cdot
      \sum_{i=0}^{j-1} \biggl|\frac{T_{a_1}' (x_i (a_1))}{T_{a_2}' (x_i
      (a_2))} - 1 \biggr| \\
    \nonumber
    &\leq 1 + e^{(C + 2 a_2) \sum_{i=0}^{j-1} C |x_i(\omega)|
      i^{\kappa_0}} \cdot (C + 2 a_2) \sum_{i=0}^{j-1} C
      |x_i(\omega)| i^{\kappa_0} \\
    \label{eq.standarddist}
    &\leq 1 + e^{C'' j^{2\kappa_0+1+\kappa_2} |x_{j}(\omega)|} \cdot
      C'' j^{2\kappa_0+1+\kappa_2} |x_{j}(\omega)|  
      \,,\,\,\forall j\leq n\, , 
  \end{align}
  where we used \eqref{eq.36} and \eqref{eq:polapproach} (the
  first bound of \eqref{monot} if $i<N_0$) in the second
  inequality, and \eqref{need2} in the last inequality.  
  Setting $M_0:=4\kappa_0+3+2\kappa_2$, if  \eqref{need} holds
  for $\omega$, then \eqref{need2} gives for all $N_0\le j \leq n$
  \begin{align*}
  C'' j^{2\kappa_0 +1+\kappa_2} |x_{j}(\omega)|  
    & \leq  C'' j^{2\kappa_0+1+\kappa_2} C' j^{\kappa_0+1+\kappa_2}|x_{n}(\omega)|\\
      &\leq C' \frac{j^{3\kappa_0+3+2\kappa_2}}{n^{M_0/(1-\alpha)}} 
\le C'''
       \, .
  \end{align*}
  This proves \eqref{eq.37}. Similarly,
  $ \Bigl|\frac{
    (T_{a_2}^j)'(x_0(a_2))}{(T_{a_1}^j)'(x_0(a_1))}\Bigr| \leq
  1+C'' e^{C''} j^{2\kappa_0+1+\kappa_2} |x_{j}(\om)|$. Therefore,
  \[
   \biggl  |\frac1{(T_{a_1}^j)'(x_0(a_1))}-\frac1{(T_{a_2}^j)'(x_0(a_2))}\biggr |
    \leq C'''\frac{j^{2\kappa_0+1+\kappa_2}
      |x_{j}(\omega)|}{|(T_{a_1}^j)'(x_0(a_1))|}\,.
  \]
  We can then adapt the end of the proof of \cite[(31)]{DS}:
  Comparing each term on the right-hand side of
  \begin{equation*}
    \frac{x_n'(a)}{(T_a^n)'(x_0(a))}
    =x_0'(a)+\sum_{j=1}^n\frac{(\partial_aT_a)(x_{j-1}(a))}{(T_a^j)'(x_0(a))}\, ,\,
    \forall a\in\tilde \omega\in \PP_n\,,
  \end{equation*}
  for $a=a_1$ and $a=a_2$, we find, since
  $x_0'(a)=\partial_ac_1(a)=1/4$, and
  \[
    |\partial_aT_a|_{a_1}(x_{j-1}(a_1))-\partial_aT_a|_{a_2}(x_{j-1}(a_2))|
    \leq |x_{j-1}(a_1)-x_{j-1}(a_2)|\leq |x_{j-1}(\om)|\, ,
  \]
  recalling \eqref{eq:CE}, and applying  \eqref{need2}
  and then \eqref{need} for $M_0=4\kappa_0+3+2\kappa_2$,
  \begin{align*}
    \biggl| \frac{x_n'(a_1)}{(T_{a_1}^n)'(x_0(a_1))} \biggr|
    &\leq \biggl|\frac{x_n'(a_2)}{(T_{a_2}^n)'(x_0(a_2))} \biggr|
      + \hat C |a_1-a_2| + \hat C \sum_{j=1}^n\frac{j^{2\kappa_0+1+\kappa_2}}{\lambda_{CE}^{j}} |x_{j}(\omega)|
    \\
    &\leq \biggl |\frac{x_n'(a_2)}{(T_{a_2}^n)'(x_0(a_2))} \biggr|
      +\bar C C' n^{4\kappa_0+3+2\kappa_2}|x_{n}(\omega)|\\
    &\leq \biggl|\frac{x_n'(a_2)}{(T_{a_2}^n)'(x_0(a_2))} \biggr| + \tilde
      C|x_{n}(\omega)|^\alpha \, .
  \end{align*}
  Finally, we have, using \eqref{l.transversality} (which plays
  the role of \cite[Lemma~2.4]{DS}),
  \begin{align*}
    \biggl| \frac{x_n'(a_1)}{x_n'(a_2)} \biggr|
    & \leq \biggl| \frac{(T_{a_1}^n)' (x_0(a_1))}{(T_{a_2}^n)' (x_0
      (a_2))} \biggr| \biggl( 1 + \tilde C |x_{n} (\omega)|^\alpha
      \frac{|(T_{a_2}^n)'(x_0(a_2))|}{|x_n'(a_2)|} \biggr) \\
    & \leq 1 + C \tilde C |x_{n} (\omega)|^\alpha\,. \qedhere
  \end{align*}
\end{proof}

\subsection{Uniform Decorrelation and H\"older Response}\label{s.pud}
The maps $x_j$ are not the iterates
of a fixed dynamical system admitting an invariant measure.  To
exploit statistical information on the iterates of the mixing CE map
$(T_{a_0},\mu_{a_0})$, we will ``switch locally'' from $x_j$ to
$T_{a_0}^j$ (see Lemma~\ref{l.switch}), using that any
$a\in \Omega_{*}$ satisfies\footnote{The factor
  $\|\varphi\|_{L_1(dm)}$ in the right-hand side of
  \cite[Prop.~4.3]{DS} is replaced in Proposition~\ref{p.udc} by
  $\|\varphi\|_{L^1(d\mu_a)}\le \|\varphi\|_{L^\infty(dm)}$. This does
  not impact \cite[p.~36, use of Prop.~4.3]{DS}.}  the following
uniform decorrelation result for H\"older continuous observables.  For
$q>1$ and $s\in [0, 1/q)$, we denote by $H^s_q(I)=F^s_{q,2}(I)$ the Sobolev space
of functions of differentiability $s$ and integrability $q$ supported
in $I$ (see \cite{RS}).

\begin{proposition}[Uniform Decay of Correlations]
  \label{p.udc}
  For any $s>0$ and $q>1$, there exist $C<\infty$ and
  $\rho^s_q<1$ such that, for all $\varphi\in H^s_q(I)$,
  $\psi \in L^\infty(dm)$, $a\in \Omega_*(a_*, \kappa_0)$
  \[
    \biggl| \int_0^1 \varphi (\psi\circ T_a^n) \, dm -
    \int_0^1\varphi \, dm \int_0^1\psi \, d\mu_a \biggr| \leq C
    \|\varphi\|_{H^s_q}\|\psi\|_{L^1(d\mu_a)}(\rho^s_q)^n\,
    ,\forall n\geq 1\, .
  \]
  For any $\varpi >0$, there exist $C<\infty$ and $\rho_\varpi<1$
  such that, for all $\varphi\in C^\varpi$,
  $\psi \in L^\infty(dm)$, $a\in \Omega_*(a_*, \kappa_0)$
  \[
    \biggl| \int_0^1 \varphi (\psi\circ T_a^n)\, d\mu_{a } -
    \int_0^1 \varphi \, d\mu_{a}\, \int_0^1 \psi\, d\mu_{a}
    \biggr| \le C \|\varphi\|_\varpi \|\psi\|_{L^1(d\mu_a)}
    (\rho_\varpi)^{n}\, ,\, \forall n\ge 1\, .
  \]
\end{proposition}

We also use H\"older bounds on $a\mapsto \mu_a$ as a distribution
(in Lemma~\ref{l.varregularity}):

\begin{proposition}[Fractional Response]
  \label{p.udc'}
  For any $\Theta \in (0,1/2)$, there exists $C$ such that for
  all $\varphi \in C^{1/2}$
  \begin{equation}\label{eq.densitydiff}
    \biggl| \int \varphi \, d\mu_a - \int \varphi \, d\mu_{a'}
    \biggr| \leq C |a-a'|^\Theta \|\varphi\|_{1/2}\, , \quad
    \forall a, a'\in \Omega_* (a_*,\kappa_0)\, .
  \end{equation}
\end{proposition}

Our proof of Proposition~\ref{p.udc} uses the following facts.

\begin{sublem}\label{fact} 
  For any $a\in \Omega_{BC}$, the density $h_a$ of $\mu_a$ lies
  in $H^s_q(I)$ for all $s\in [0,1/2)$ and $q \in (1, 2/(1+2s))$.
  In addition, for 
 any $(H_0,\kappa_0)$
  polynomially recurrent $a_*$, there exists
  $C_{s,q,a_*}<\infty$ such that
  \[
    \sup_{a\in \Omega_*{a_*,\kappa_0}}
    \|h_a\|_{H^s_q(I)}\le C_{s,q,a_*}\, .
  \]
\end{sublem}

\begin{proof}
  In the Misiurewicz case, the first claim is \cite[Theorem~10]{Sed} ,
  using Ruelle's \cite[Theorem~9, Remark~16.a]{Ru} decomposition of
  $h_a$ into the sum of a $C^1$ function and an exponentially decaying
  sum of ``spikes'' $x\mapsto |x-c_k(a)|^{-1/2}$ and square root
  singularities $x\mapsto |x-c_k(a)|^{1/2}$. For a general
  $a \in \Omega_{BC}$, set
  $ T^{-k}_{a,\varsigma}:=(T_a^k|_{U_{k,a,\varsigma}})^{-1}$, for
  $k\ge 1$ and $\varsigma\in \pm$, where $U_{k,a,\varsigma}$ is the
  monotonicity interval of $T_a^k$ containing $c$, located to the
  right of $c$ for $\varsigma=+$, to the left of $c$ for
  $\varsigma=-$. Then, since we assumed $\lambda_{CE}>e^{ 14 \alpha_{BC}}$ in the proof of Proposition~\ref{keyprop},   use
  \cite[Prop~2.7]{BS1} that there exist a $C^1$ function
  $\psi_a \colon I \to \real_+$ and
  $C^\infty$ functions $\Xi^k_{a,\pm} \colon [0,1] \to [0,1]$
  supported in a neighbourhood of $c_k(a)$ in  $T^k_a(U_{k,a,\pm})$, such that
  \begin{equation}\label{smoothrho'}
    h_a(x)=\psi_a(x) + \sum_{k = 1}^{\infty}\sum_{\varsigma\in
      \{+,-\}}\chi_{k,a}(x)
    \frac{\Xi^k_{a,\varsigma}(T^{-k}_{a,\varsigma}(x))
      \psi_a(T^{-k}_{a,\varsigma}(x))
    }{|(T_a^{k})'(T^{-k}_{a,\varsigma}(x))|} \, ,
  \end{equation}
  where $\chi_{k,a}(x)=1_{\pm x<\pm c_k(a)}$ if $\pm T_a^k$ has a
  local maximum at $c$. Setting  $\Psi:=\Xi^k_{a,\varsigma}
    \cdot  \psi_a$, we find $C^1$ functions $\Psi_{k,\ell}$, for $\ell=1,2,3$,
with
 \begin{multline}\label{glu}
   \frac{
\Psi (T^{-k}_{a,\varsigma}(x))
   }{|(T_a^{k})'(T^{-k}_{a,\varsigma}(x))|} 
   = 
 \frac{\Psi_{k,1}(x)}{|x-c_k(a)|^{1/2} }
   + \Psi_{k,2}(x) |x-c_k(a)|^{1/2} + \Psi_{k,3}(x)
\, ,
  \end{multline}
 for any $x\in \supp (\chi_{k,a})$. Finally, use \cite[Lemmas~11--12]{Sed}.
  
  For the second claim, it is convenient to use an alternative
  decomposition of $h_a$. First recall that \cite[Cor 1.6]{BBS} gives
  a set $\Omega_\mathrm{slow}$ of full measure in the set of mixing CE
  parameters such that, for any $\tilde{a} \in \Omega_\mathrm{slow}$
  and each $\kappa_0>1$, there exist $H_0\ge 1$ and a set
  $\Delta_0(\tilde{a}, \kappa_0)\subset 
  \Omega_\mathrm{slow}$ of $(H_0,\kappa_0)$-polynomially recurrent (and thus transversal)
  parameters, with $\tilde{a}$ as a Lebesgue density point, such that
  Proposition~\ref{p.udc} holds for all $a \in \Delta_0$.  (It is
  unknown whether $\tilde a\in \Delta_0$.) The proof involves
  constructing a tower for each parameter in $\Delta_0$.
  We claim that, up to reducing the value of
    $\epsilon$ in the proof of Proposition~\ref{keyprop}, we can
    replace $\tilde a$ by $a_*$ and $\Delta_0$ by
    $\Omega_*(a_*, \kappa_0)$. Indeed, $\Delta_0$ was constructed in
    \cite[Prop.~2.1]{BBS}, and it suffices to observe that the
    required uniformity in constants is satisfied by \eqref{eq:CE} and
    \eqref{eq:polapproach}, while \cite[(8) and (7)]{BBS} are exactly
    \cite[V.(6.1), V.(6.2) in Prop.~V.6.1]{DMS}.

  Let then 
  \[
    \Pi_a (\hat \psi)(x)=\sum_{j \ge 0, \varsigma \in\pm} \frac{
      \lambda^j} {|(T_a^j)'(T^{-j}_{a,\varsigma}(x))|}
    \psi_j(T^{-j}_{a,\varsigma}(x))\, ,
  \]
  (for a suitable $\lambda>1$) be the projection from the tower
  with polynomial recurrence
  used in \cite{BBS}, and let $\hat \LL_a$ be the lift
  $\LL_a \Pi_a=\Pi_a \hat \LL_a$ of the transfer operator
  $\LL_a \varphi(x)=\sum_{T_a(y)=x}
  \varphi(y)/|T'_a(y)|$. Then\footnote{\cite[(66)]{BBS} gives uniform
    Lasota--Yorke estimates.  \cite[Lemma~3.8, Lemma~4.5, Lemma~4.6,
    Prop.~4.1]{BBS} give the weak norm bounds needed for
    Keller--Liverani \cite{KeLi}.} 
  there exist $C<\infty$ and $\theta<1$ such that, letting
  $\|\cdot\|'_a$ be the norm of the Sobolev space $\BB_a^{W^1_1}$
  of \cite{BBS},
  \begin{equation}\label{tran}
    \|\hat \LL_a^n(\hat \varphi)-\hat h_a \hat \nu(\hat
    \varphi)\|'_a\le C \|\hat \varphi\|'_a \theta^n\, , \qquad \forall 
    \hat \varphi \in \BB_a^{W^1_1}\, , \quad \forall a\in
    \Omega_{\tilde{a}}\, ,
  \end{equation}
  where $\hat h_a$ is the fixed point\footnote{The fixed point
    property determines $\hat h_a$ by its value on the level zero
    of the tower.} of $\hat \LL_a$ on $\BB_a^{W^1_1}$ normalised
  by $\int \Pi_a \hat h_a \, dm = 1$, while $\hat \nu$, the
  nonnegative measure whose density with respect to Lebesgue in
  the level $j$ of the tower is $\lambda^j$, is the fixed point
  of the dual of $\hat \LL_a$ (see \cite[(85)]{BS1}, note that
  $\nu(\hat h_a)=1$ is automatic).  Since $\Pi_a\hat h_a=h_a$ and
  the $W^1_1$ norm dominates any $H^s_q$ norm on $I$ if
  $s\in [0,1)$ and $1< q<1/s$ (by the Sobolev embedding, more
  precisely \cite[Chapter~2]{RS} the bounded
  inclusions
  $ W^1_1 \subset W^\sigma_1=F^{\sigma}_{1,1}\subset
  F^{\sigma}_{1,2}\subset F^{s}_{q,2}=H^s_q$, if
  $\sigma=1+s-1/q\in (0,1)$ and $q\in (1,\infty)$), the
  decomposition \eqref{glu} combined with the uniform bound
  \eqref{tran} (for $\hat \varphi$ vanishing on all levels
  $\geq 1$ and constant on level zero of the tower, with
  $\hat \nu(\varphi)=1$) gives the second claim of the sublemma,
  using again \cite[Lemmas~11--12]{Sed}.
\end{proof}

\begin{proof}[Proof of Proposition~\ref{p.udc}]
  Recall from the proof of Proposition~\ref{keyprop} that we have
  $\lambda_{CE}>e^{14 \alpha_{BC}}$.  By mollification, it is
  enough to prove both bounds for $C^1$ functions $\varphi$.  It
  is in fact enough to show the first bound for $\varphi\in C^1$:
  Indeed, again by mollification (see e.g.\ the proof of
  \cite[Lemma~14]{Sed}), if the first bound holds for
  $\varphi\in C^1$, then it holds for any $\varphi \in H^s_q(I)$
  with $q>1$ and $s>0$.  Therefore, since the density $h_a$ of
  $\mu_a$ lies in $H^s_q(I)$ for all $s\in (0,1/2)$ and
  $q \in (1, 2/(1+2s))$ by Sublemma~\ref{fact} (with norm
  uniformly bounded in $a$), the second bound
  follows from the first bound for $\varphi \in C^1$ (using that
  $C^1$ functions are bounded multipliers on $H^s_q$).

  Next, we observed in the proof of Sublemma~\ref{fact} that we
  can replace the set called $\Delta_0$ in
  \cite[Cor.~1.6]{BBS} by $\Omega_*(a_*, \kappa_0)$. The
  first bound for Lipschitz continuous $\varphi$ thus follows
  from the second assertion of \cite[Cor.~1.6]{BBS}, since
  $\Omega_*\subset [a_\mathrm{mix}, 4)$.  Indeed, note first that
  $a$ is topologically mixing if and only if its renormalisation
  period $P_{a}$ is equal to one.  Second, observe that the
  constant $C_{\varphi, \psi}$ in the second claim of
  \cite[Cor.~1.6]{BBS} can be replaced by
  $C \|\varphi\|_\varpi\|\psi\|_{L^1(d\mu_a)}$, for a constant
  $C$ uniform in $a$ in view of \cite[Lemma~4.5, Lemma~4.6]{BBS}
  and the principle of uniform boundedness.  More precisely,
  using the notation from the proof of Sublemma~\ref{fact}, we
  have
  \begin{align*}
    &\int (\psi \circ T_a^n) \varphi \, dm = \int \psi \Pi_a (\hat
      \LL_a^n(\hat \varphi)) \, dm \qquad \qquad \mbox{ if } \Pi_a(\hat
      \varphi)=\varphi \, , \\
    & \int (\psi \circ T_a^n) \varphi h_a \,
      dm= \int \psi \Pi_a(\hat \LL_a^n(\hat \varphi_a)) \, dm \qquad\,\,
      \mbox{ if } \Pi_a(\hat \varphi_a)=\varphi h_a \, .
  \end{align*}
  Since\footnote{\label{a=4}The Banach space of \cite{BBS}
    requires that the function on level zero of the tower be
    supported in $(0,1)$, so this proof cannot cover the case
    $a=4$. } $\Pi_a\hat h_a=h_a$, any Lipschitz continuous
  $\varphi$ can be written as $\Pi_a (\hat \varphi)$ (take
  $\hat \varphi_0= \varphi$ on the level zero, and
  $\hat \varphi_j\equiv 0$ on levels $j\ge 1$) such that, on the
  one hand, $\|\hat \varphi\|_a'\le C\|\varphi\|_1$ uniformly in
  $a$, and, on the other hand,
  $\hat \nu(\hat \varphi) = \int \varphi \, dm$, we conclude by
  applying \eqref{tran} from the proof of Sublemma~\ref{fact}.
\end{proof}

\begin{proof}[Proof of Proposition~\ref{p.udc'}]
  If $a=a_*$, the bound is an immediate consequence of the first
  claim of \cite[Cor.~1.6]{BBS}, since we can
  replace the set denoted $\Delta_0$ there by
    $\Omega_*(a_*, \kappa_0)$, as observed in the proof of
    Sublemma~\ref{fact} and used in the proof of
    Proposition~\ref{p.udc}.  If $a\ne a_*$, the uniformity of
    the constants given by Proposition~\ref{keyprop} ensures that
    we may construct the reference tower in \cite{BBS} at $a$
    (instead of $a_*$), viewing $a'$ as a perturbation of $a$. 
\end{proof}

\subsection{H\"older Regularity of the Variance
  $\sigma_a(\varphi)$}
\label{regvar}
Propositions~\ref{p.udc} and~\ref{p.udc'} will imply the
following regularity of $a\mapsto\sigma_a(\varphi)$ on
$\Omega_{*}$.
      
\begin{lemma}[Regularity of $\sigma_a(\varphi)$]
  \label{l.varregularity}
  For any $\varpi \in (0,1]$, there exist
  $\theta \in (0,\min\{1/2, \varpi\})$ and $C<\infty$ such that,
  for each $\varphi \in C^\varpi$ with $\sigma_{a_*}(\varphi)>0$
  there exists $\epsilon_\varphi>0$ such that
  \[
    C_{\epsilon_\varphi}(\varphi):=\inf_{a\in \Omega_{a_*}\cap
      [a_*-\epsilon_\varphi, a_*+\epsilon_\varphi]}
    \sigma_{a}(\varphi)>0\, ,
  \]
  and such that for all
  $a,a'\in \Omega_{*}(a_*,\kappa_0)\cap [a_*-\epsilon_\varphi,
  a_*+\epsilon_\varphi]$ we have
  \begin{equation}
    \label{eq.varregularity}
    |\sigma_a(\varphi)-\sigma_{a'}(\varphi)|\le\frac{ C}{2
      C_\epsilon(\varphi)} \|\varphi\|_\varpi |a-a'|^\theta\,.
  \end{equation}
\end{lemma}

\begin{proof}
  Let $k_0>1$ be a large integer to be chosen at the end of the
  proof.  By the second claim of Proposition~\ref{p.udc}, there
  exist $\rho=\rho_\varpi<1$ and $C_0$ such that
  \begin{multline*}
    \sum_{k> k_0} \biggl| \int \biggl( \varphi-\int\varphi \,
    d\mu_{a} \biggr) \cdot \biggl( \biggl( \varphi - \int\varphi
    \, d\mu_{a} \biggr) \circ T_{ a}^k \biggr) \, d\mu_{a}
    \biggr| \\
    \le C_0 \|\varphi\|_\varpi^2 \cdot
    \frac{\rho^{k_0}}{1-\rho}\, , \qquad \forall k_0\ge 1\, ,
    \quad \forall a \in \Omega_{a_*} \, , \quad \forall \varphi
    \in C^\varpi\,.
  \end{multline*}
Set $A_a=\int \varphi \, d\mu_a$.
  Since
  $\int \bigl ((\varphi-A_a)\circ T_a^k \bigr )(\varphi -A_a) \,
  d\mu_a = \int (\varphi\circ T_a^k) \varphi \, d\mu_a -A_a^2$,
  we have
  \begin{align*}
    |\sigma_a(\varphi)^2-\sigma_{a'}(\varphi)^2|
    \leq 2 & \sum_{k=0}^{k_0-1} \biggl |\int\varphi(\varphi\circ
          T_a^k) \,d\mu_a - \int\varphi(\varphi\circ T_{a'}^k) \,
          d\mu_{a} \biggr | \\
        & + 2\sum_{k=0}^{k_0-1} \biggl |\int\varphi(\varphi\circ
          T_{a'}^k) \, d\mu_a - \int\varphi(\varphi\circ
          T_{a'}^k) \, d\mu_{a'} \biggr | \\
        & + 2\sum_{k=0}^{k_0-1}\biggl | \biggl( \int \varphi \,
          d\mu_a \biggr)^2 - \biggl( \int \varphi \, d\mu_{a'}
          \biggr)^2 \biggr | \\
        & + 4 C_0 \|\varphi\|^2_\varpi\frac{\rho^{k_0}}{1-\rho}\,
          , \qquad \forall k_0\ge 1 \,.
  \end{align*}
  Assume for a moment that $\varpi\ge 1/2$.  The
  $\varpi$-H\"older constant of
  $\varphi (\varphi \circ T_{\bar a}^k)$ (for $\bar a=a$ or $a'$)
  is bounded by $\Lambda^{ k}\|\varphi\|^2_\varpi$.  Thus,
  Proposition~\ref{p.udc'} gives for any $\Theta<1/2$ a constant
  $C_1=C_1(\Theta)$ such that for $a, a' \in \Omega_{a_*}$, and
   $\varphi\in C^\varpi$
  \begin{multline*}
    |\sigma_a(\varphi)^2-\sigma_{a'}(\varphi)^2| \leq k_0 C_1
    \|\varphi\|^2_\varpi \Lambda^{ k_0}|a-a'|^\Theta
    + C_0 \|\varphi\|^2_\varpi \frac{\rho^{k_0}}{1-\rho}\\
    + 2 \sum_{k=0}^{k_0-1} \biggl |\int\varphi(\varphi\circ T_a^k)
    \, d\mu_a - \int\varphi(\varphi\circ T_{a'}^k) \, d\mu_{a}
    \biggr |\, ,\qquad \forall k_0\geq 1 .
  \end{multline*} 
  Next, \eqref{eq.xaa} gives  that
  \begin{align*}
    \int
    &|\varphi\circ T_a^k-\varphi\circ T_{a'}^k |\, d \mu_a \leq
      \|\varphi\|_{\varpi} (C \Lambda^{k}|a-a'|)^{\varpi}\, .
  \end{align*}
  Therefore, we find
  \begin{multline}
    \label{concl}
    \bigl|\sigma_a(\varphi)^2-\sigma_{a'}(\varphi)^2 \bigr| \leq
    k_0 C_1 \|\varphi\|^2_\varpi \Lambda^{ k_0}|a-a'|^\Theta +
    4 C_0 \|\varphi\|^2_\varpi \frac{\rho^{k_0}}{1-\rho}\\
    + k_0 \|\varphi\|_{\varpi} ( C\Lambda^{k_0}|a-a'|)^{\varpi}\,
    .
  \end{multline}
  We conclude the proof for $\varpi \ge 1/2$ by dividing
  \eqref{concl} by $|a-a'|^{\theta}$, for small enough $\theta>0$
  and optimising in $k_0$, using also
  $(\sigma_a - \sigma_{a'}) (\sigma_a + \sigma_{a'}) = \sigma_a^2
  - \sigma_{a'}^2$.
	
  If $\varpi \in (0,1/2)$, mollification gives
  $\varphi_\upsilon\in C^{1/2}$ and $C_4$ such that
  \[
    \|\varphi_\upsilon\|_{1/2} \leq C_4
    \upsilon^{\varpi-1/2}\|\varphi\|_\varpi \, , \,\, \sup
    |(\varphi\circ T_{\bar a}^k)\varphi - (\varphi_\upsilon \circ
    T_{\bar a}^k)\varphi_\upsilon| \leq C_4
    \upsilon^{\varpi}\Lambda^k\|\varphi\|_\varpi \, ,
  \]
  for all small $\upsilon >0$, all $0\le k \le k_0$, and all
  $\bar a \in \Omega_{a_*}$. To conclude, optimise in
  $\upsilon=|a-a'|^{\theta_0}$ for small $\theta_0>0$, taking
  $\theta$ smaller (in particular $\theta<\varpi\theta_0$).
\end{proof}

\section{Switching Locally from the Parameter to the Phase Space}
\label{s.switch}

Let $a_*$,  $\PP_j(a_*,\kappa_0)$, and
$\Omega_*=\Omega_*(a_*,\kappa_0)$ be as in
Proposition~\ref{keyprop} for $\kappa_0 \ge 11 / (3d_1)$,
and fix $\varpi\in (0,1)$.  This
section is devoted to Proposition~\ref{p.mainestimate}, the main
estimate (analogous to \cite[Prop.~5.1]{DS}) towards a law of
large numbers for the squares of the blocks which will be defined
in Section~\ref{s.asip} (see Lemma~\ref{l.LLNy}).

From now on, fix $\varpi\in (0,1)$ and a $\varpi$-H\"older
continuous function $\varphi \colon I\to\real$, recalling
$\varphi_a$, $\sigma_a(\varphi)$ from \eqref{eq.varphia},
\eqref{eq.sigma}, and assume $\sigma_{a_*}(\varphi)>0$.
Lemma~\ref{l.varregularity} gives $\epsilon_\varphi>0$ such that
\begin{equation}\label{omegaphi}
  \sigma_a(\varphi)>0 \,, \quad \forall
  a\in\Omega_*^\varphi:=\Omega_*\cap [a_*-\epsilon_\varphi,
    a_*+\epsilon_\varphi]\, .
\end{equation}
If $\epsilon_\varphi<\epsilon$, we replace $\Omega_*$ by
$\Omega_*^\varphi$ by replacing $\epsilon$ in the proof of
Proposition~\ref{keyprop} with $\epsilon_\varphi$. (This is
harmless as it can only improve the constants.)

\begin{remark}[$\theta$-H\"older Whitney Extensions of
  $\varphi_a$ and $\xi_n(a)$]
  \label{Whitney}
  By Proposition~\ref{p.udc'}, the function
  $a \mapsto \int \varphi \, d \mu_a$ is $\Theta$-H\"{o}lder
  continuous on $\Omega_*$ for any $\Theta <1/2$. By
  Lemma~\ref{l.varregularity}, the function
  $a \mapsto \sigma_a(\varphi)\ge 0$ is $\theta$-H\"older
  continuous on $\Omega_*$ for some
  $\theta<\min \{1/2, \varpi\}$, and uniformly bounded away from
  zero on $\Omega_*^\varphi$.  Taking $\Theta \ge \theta$, the
  map
  $a \mapsto \varphi_a(u) = (\varphi(u)-\int \varphi
  d\mu_a)/\sigma_a$ is $\theta$-H\"older continuous on
  $\Omega_*^\varphi$ uniformly in $u\in I$.  By the Whitney
  extension theorem, we extend each map $a \mapsto \varphi_a(u)$
  to a $\theta$-H\"older continuous map on
  $[a_*-\epsilon_\varphi, a_*+\epsilon_\varphi]$, uniformly in
  $u\in I$. In addition, there exists $\widetilde C < \infty$
  such that
  \begin{equation}
    \label{eq.alphal1infty}
    \|\varphi_a\|_\infty\le \|\varphi_a\|_{\varpi}
    \le \widetilde C\|\varphi\|_{\varpi}\,, \qquad \forall a\in
    [a_*-\epsilon_\varphi, a_*+\epsilon_\varphi] \,.
  \end{equation}
  Then, using \eqref{eq.xaa}, we may extend each map
  $a \mapsto \xi_n(a)=\varphi_a(T_a^{n+1}(c))$ to a
  $\theta$-H\"older continuous map on
  $[a_*-\epsilon_\varphi, a_*+\epsilon_\varphi]$, with
  $\theta$-H\"older constant bounded by
  $C\Lambda^{\theta (n+1)}$: Indeed, recalling
  $x_n(a) = T_a^{n+1}(c)$, just decompose
  \begin{equation}\label{decompose}
    \xi_n(a)-\xi_n(a')=
    \varphi_a(x_n(a))-\varphi_{a'}(x_n(a))
    +\varphi_{a'}(T_a^{n+1}(c))-\varphi_{a'}(T_{a'}^{n+1}(c))\, .
  \end{equation} 
\end{remark}

Fix $\alpha \in (0,1)$ such that (in view of the use
  of \eqref{need} in Corollary~\ref{c.switch})
\begin{equation}\label{inview}
   \frac{M_0}{1 - \alpha} \leq
  \frac{3}{\alpha}\, .
\end{equation}
Fix $q>1$ and $0<s<\min\{\varpi, 1/q\}$, and let 
\begin{equation}
  \label{eq.lambda0}
  \lambda_0 = \min (\lambda_{CE}^{\theta}, \rho^{-1/2}) > 1\,,
\end{equation}
where $\lambda_{CE}>1$ is given by \eqref{eq:CE}, while
$\rho=\max \{\rho^s_q, \rho_\varpi\}<1$ is given by
Proposition~\ref{p.udc}, and $\theta\in (0,\min\{1/2,\varpi\})$
is given by Lemma~\ref{l.varregularity}.  Finally, recalling
$\Lambda$ from \eqref{eq.la}, let $\eta\in (0,1/2)$ be so small
that
\begin{equation}
  \label{eq.eta}
  \Bigl( \frac{2 \Lambda}{\lambda_{CE}} \Bigr)^{\eta} \leq
  \lambda_0 < \frac{\lambda_{CE}^\varpi }{\Lambda^{\eta \varpi}
  }\,.
\end{equation}

Define the expectation $E(\psi)$ of $\psi\in L^\infty(\Omega_*^\varphi)$
by\footnote{We restrict to the Cantor set
    $\Omega_*^\varphi$ here and thus in
    \eqref{eq.mainestimate}. The bound \eqref{lem:measureinomega}
    is used in the proof of \eqref{eq.mainestimate} (but not for
    \eqref{eq.mainestimatelocal}, Lemma~\ref{l.switch}, or
    Corollary~\ref{c.switch}).}
\begin{equation}\label{expect}
  E(\psi):= \frac{1}{m (\Omega_*^\varphi)}
  \int_{\Omega_*^\varphi} \psi \, dm \, .
\end{equation}
The following result is the key estimate on
$\xi_j(a)=\varphi_a (T_a^{j+1}(c))$.

\begin{proposition}
  \label{p.mainestimate}
  There exist $C_\varphi<\infty$ and $K<\infty$
  such that  
  \begin{equation}
    \label{eq.mainestimate}
    \Biggl| E \Biggl(\sum_{j=k}^{k+n-1} \xi_j \Biggr)^2 -n \Biggr|
    \leq C_\varphi \, , \quad \forall k \geq \max \{2 K,
    [2/\eta]\}\, , \ \forall 1\leq n\leq \eta k/2\,,
  \end{equation}
  and, setting\footnote{\label{1/4}The stretched
      exponent $1/4$ for $v(k)$ and the lower bound can be
      replaced by any number in $(0,1)$, without changing the
      statements, up to adjusting intermediate constants.}
  $v(k) = [k - k^{1/4}]$, for every nontrivial interval
  $\om \subset \tilde \om\in\PP_{v(k)}$ with
  $\omega\cap \Omega_*^\varphi\ne \emptyset$ and
  $\lambda_0^{-k^{1/4}} \leq |x_{v(k)} (\omega)| \leq
  {v(k)}^{-3/\alpha}$, we have
  \begin{equation}
    \label{eq.mainestimatelocal}
    \biggl| \frac 1 {|\omega|} \int_\omega \biggl(
    \sum_{j=k}^{k+n-1} \xi_j \biggr)^2 \, dm - n \biggr| \leq
    C_\varphi \, , \quad \forall k \geq [2/\eta]\, , \ \forall
    1 \leq n \leq \eta k / 2 \, ,
  \end{equation}
  and,  for any  sequence $\Psi_k$ with
  $C_\Psi:=\sup_k k^{-8/3}\sup |\Psi_k|<\infty$, and for  any refinement $\QQ_{v(k)}$ of $\mathcal{P}_{v(k)}$ such\footnote{We  have
 $\lambda_0^{-k^{1/4}} \leq |x_v (\omega)|$
  for all  $\omega \in \PP_{v(k)}$ by \eqref{theproof}.} that
  $\lambda_0^{-k^{1/4}} \leq |x_v (\omega)| \leq v^{-3/\alpha}$
  for  all $\omega \in \QQ_{v(k)}$, setting
\begin{equation}
  \label{newdef}
\QQ_{*,v(k)}:=\{\omega \in \QQ_{v(k)} \mid \omega\cap \Omega_*^\varphi \ne \emptyset \} \, , \qquad
  \Omega_{*,v(k)}^\QQ=\cup_{\omega \in \QQ_{*,v(k)}}\omega\,,
\end{equation}
we have
  \begin{equation}
    \label{eq.extraestimate}
    \Biggl| E( \Psi_k) -
    \frac{1}{|\Omega_{*,v(k)}^\QQ|} \int_{\Omega_{*,v(k)}^\QQ}
    \Psi_k\, dm  \biggr| \leq C_\Psi C_\varphi \, ,\,\, \forall k
    \geq [2/\eta] \, .
  \end{equation}
\end{proposition}
 
Proposition~\ref{p.mainestimate} is proved in
Section~\ref{keysection}. Like for its analogue
\cite[Prop.~5.1]{DS}, the first step will be to show the
local estimate \eqref{eq.mainestimatelocal} using
Lemma~\ref{l.switch} through its Corollary~\ref{c.switch} (the
analogues of \cite[Lemma~5.3, Cor.~5.5]{DS}).

\begin{lemma}[Switching Locally from Parameter to Phase Space]
  \label{l.switch}
  Fix $\ell_0 \in \{1,2,3, 4\}$. There exists $C<\infty$ such
  that we have, for any integers
  \[
    n \leq n_i \leq n + \eta n\, , \qquad 1 \leq i \leq \ell_0 \,
    ,
  \]
  for every $\tilde{\omega} \in \PP_n$ and each
  nontrivial interval $\omega \subset \tilde{\omega}$
  with $\omega\cap \Omega_*^\varphi\ne \emptyset$,
  \begin{align}
    \label{eq.switch}
    \int\limits_{x_n (\omega)} \biggl|
    \prod_{\ell=1}^{\ell_0} \xi_{n_\ell}
    (x_n|_\omega^{-1} (y)) &- \prod_{\ell=1}^{\ell_0}
    \varphi_{a_0}
    (T_{a_0}^{n_\ell-n}(y)) \biggr| \,
      dy \\
    \nonumber
    &\leq C \lambda_0^{-n} |x_n (\omega)|\, ,
      \qquad \forall
      a_0 \in \omega \cap \Omega_*^\varphi\,.
  \end{align}
\end{lemma}

\begin{corollary}
  \label{c.switch}
  There exists $C_3 >1$ such that, for $\ell_0$, $n$,
  $ n_1, \ldots, n_{\ell_0}$, and $\omega$ as in
  Lemma~\ref{l.switch}, if, in addition,
  $ |x_n (\omega)| \le n^{-3/\alpha}$, then for any
  $a_0 \in \omega\cap \Omega_*^\varphi$,
  \begin{multline*}
    \biggl| \frac1{|\omega|} \int_{\omega}
    \prod_{\ell=1}^{\ell_0} \xi_{n_\ell} (a) \, da -
    \frac{1}{|x_n (\omega)|} \int_{x_n (\omega)}
    \prod_{\ell=1}^{\ell_0} \varphi_{a_0} (T_{a_0}^{n_\ell - n}
    (y)) \, dy \biggr| \\ \leq C_3 ( |x_n (\omega)|^\alpha+
    \lambda_0^{-n})\,.
  \end{multline*}
\end{corollary}

\begin{proof}
Since \eqref{inview} implies  \eqref{need} for $\omega$,
the change of variables $y = x_n (a)$ on $\omega$, combined with 
  the distortion estimate \eqref{eq.distortion1}, gives
   \begin{align}
     \nonumber
     \biggl| \frac1{|\omega|}
     & \int_{\omega } \prod_{\ell=1}^{\ell_0} \xi_{n_\ell}(a)\,
       da-\frac{1}{|x_n (\omega )|} \int_{x_n (\omega)}
       \prod_{\ell=1}^{\ell_0} \xi_{n_\ell}(x_n|_\omega^{-1} (y))
       \, dy \biggr|
     \\ 
     &= \frac{1}{|x_n (\omega)|} \biggl| \int_{x_n (\omega)}
       \prod_{\ell=1}^{\ell_0} \xi_{n_\ell}(x_n|_\omega^{-1} (y))
       \biggl( \frac 1{ |x_n'(x_n|_\omega^{-1} y)|} \frac{|x_n
       (\omega )|}{|\omega |}-1\biggr )   dy \biggr|
       \label{eq.switchdist}\\
     \nonumber
     &\leq C \frac{|x_n (\omega)|^\alpha}{|x_n (\omega )|}
       \int_{x_n (\omega)} \prod_{\ell=1}^{\ell_0}
       |\xi_{n_\ell}(x_n|_\omega^{-1}( y)) |\, dy \, .
  \end{align}
  Since $\sup_k\|\xi_k\|_{L^\infty}<\infty$, the claim then
  follows from Lemma~\ref{l.switch}.
\end{proof}

\begin{proof}[Proof of Lemma~\ref{l.switch}.]
  For $a_0\in \omega$ as in the statement, the functions
  \[
    \tilde \varphi_\ell(y) =\tilde \varphi_{\ell,a_0}(y) =
    \varphi_{a_0} (T_{a_0}^{n_\ell - n} (y)) \, , \qquad \tilde
    \xi_\ell (y)=\tilde \xi_{\ell,\omega} (y)= \xi_{n_\ell}
    (x_n|_\omega^{-1} (y))
  \]
  with
  \[
    \xi_{n_\ell} (x_n|_\omega^{-1} (y))
    =\varphi_{x_n|_\omega^{-1} (y)} (x_{n_\ell} (x_n|_\omega^{-1}
    (y))= \varphi_{x_n|_\omega^{-1} (y)}
    (T^{n_\ell+1}_{x_n|_\omega^{-1} (y)}(c))
  \]
  are bounded on $x_n(\omega \cap \Omega^\varphi_*)$.  Decomposing
  \begin{align}
    \nonumber  |\tilde \xi_1 \tilde \xi_2 \tilde \xi_3 \tilde
    \xi_4 - \tilde \varphi_1 \tilde \varphi_2 \tilde \varphi_3
    \tilde \varphi_4| \leq |(\tilde \xi_1
    &-\tilde \varphi_1) \tilde \xi_2 \tilde
      \xi_3 \tilde \xi_4| + |\tilde \varphi_1 (\tilde \xi_2 -
      \tilde
      \varphi_2) \tilde \xi_3 \tilde \xi_4| \\
    \label{noell0}
    &+ |\tilde \varphi_1 \tilde \varphi_2 (\tilde \xi_3 - \tilde
      \varphi_3) \tilde \xi_4| + |\tilde \varphi_1 \tilde \varphi_2
      \tilde \varphi_3 (\tilde \xi_4 - \tilde \varphi_4) | \, ,
  \end{align}
  it is enough to find a uniform constant $\bar C>1$ such that
  \[
    \frac{1}{|x_n (\omega)|} \int_{x_n (\omega)} |\tilde
    \xi_{\ell,\omega} - \tilde \varphi_{\ell,a_0}| \, dy \leq
    \bar C \lambda_0^{-n}\, ,\qquad \forall a_0 \in \omega \cap
    \Omega_*^\varphi \, , \, 1 \leq \ell \leq \ell_0 \, .
  \]
 We will do so by showing  the pointwise estimate
  \[
    |\tilde \xi_{\ell,\omega} (y) - \tilde \varphi_{\ell,a_0}
    (y)| \leq \bar C \lambda_0^{-n} \, , \qquad \forall y \in x_n
    (\omega)\, ,\,\forall a_0 \in \omega \cap \Omega_*^\varphi \,
    , \, 1 \leq \ell \leq \ell_0 \, .
  \]
  For $a= x_n|_\omega^{-1} (y)$, we decompose
  \begin{align}
    \nonumber
    \tilde \xi_{\ell,\omega} (y) &- \tilde \varphi_{\ell,a_0} (y) =
    \xi_{n_\ell} (a) - \varphi_{a_0} (T_{a_0}^{n_\ell - n} (x_n (a)))
    \\
    \nonumber
    &= \varphi_a(x_{n_\ell} (a)) - \varphi_{a_0} (T_{a_0}^{n_\ell -
      n} (x_n (a)))\\
    \label{rhs2}
    &= \varphi_a(x_{n_\ell} (a)) - \varphi_{a_0}
    (x_{n_\ell} (a)) + \varphi_{a_0} (x_{n_\ell} (a)) - \varphi_{a_0}
    (T_{a_0}^{n_\ell - n} (x_n (a))) \,.
  \end{align}
  Using Remark~\ref{Whitney}, there exists $C$, independent of
  $n_\ell$, such that
  \begin{equation}\label{thetagamma}
    |\varphi_{a} (x_{n_\ell} (a))- \varphi_{a_0} (x_{n_\ell} (a)) 
    | \leq C |\omega|^\theta \, , \qquad \forall \{a,
    a_0\}\subset \omega\, .
  \end{equation}
  Hence, using our choice \eqref{eq.lambda0} of
  $\lambda_0$, and since $|\omega| \leq C \lambda_{CE}^{-n}$
  by \eqref{small}, we get
  \begin{equation}
    \label{eq:xn-integral-part1}
    |\varphi_{a} (x_{n_\ell} (a))- \varphi_{a_0} (x_{n_\ell} (a)) |
    \leq C |\omega|^\theta \leq C \lambda_0^{-n} .
  \end{equation}
  For the last two terms in the right-hand side of \eqref{rhs2}
  note that since $a=x_n|_\omega^{-1} (y)$ implies
  $x_{n_\ell}(a) = T^{n_\ell + 1}_a(c) = T^{n_\ell-n}_a
  (T^{n+1}_a (c)) = T^{n_\ell-n}_a(y)$,
  we have, using \eqref{eq.xaa},
  \begin{multline}\label{??}
    \bigl| x_{n_\ell} (x_n|_\omega^{-1} (y)) - T_{a_0}^{n_\ell -
      n}(y) \bigr|
    =|T_a^{n_\ell - n} (y) - T_{a_0}^{n_\ell - n} (y)| \\
    \leq C \Lambda^{n_\ell - n} |a - a_0| \leq C \Lambda^{n_\ell
      - n}|\omega| \, , \qquad \forall y \in x_n(\omega)\, .
  \end{multline}
  Then, since $n_\ell-n\le \eta n$, our choice of $\lambda_0$,
  $\eta$, with \eqref{eq.alphal1infty} at $a=a_0$
  give\footnote{We do not need the analogue of
      Sublemma~5.4 from \cite{DS} here.}
  \begin{align}
    \label{eq:xn-integral-part2}
    |\varphi_{a_0} (x_{n_\ell} (a)) &- \varphi_{a_0}
    (T_{a_0}^{n_\ell - n} (x_n (a)))| 
\\    \nonumber  &
= \bigl| \varphi_{a_0}
    (x_{n_\ell} (x_n|_\omega^{-1} (y))) - \varphi_{a_0}
      (T_{a_0}^{n_\ell - n}(y)) \bigr| 
\\    \nonumber   &
 \leq  C \tilde C  
                  \Lambda^{  (n_\ell - n)\varpi}
                  |\omega|^{\varpi} 
                  \leq C \tilde C \Lambda^{
                  \varpi \eta n}
                  |\omega|^{\varpi}
                  \leq  C \tilde C \lambda_0^{-n} \, ,
  \end{align}
  using again in the last inequality that
  $|\omega| \leq C \lambda_{CE}^{-n}$ from \eqref{small}.  We
  conclude by combining \eqref{eq:xn-integral-part1} and
  \eqref{eq:xn-integral-part2} into \eqref{rhs2}.
\end{proof}

\subsection{Proof of Proposition~\ref{p.mainestimate}}\label{keysection}

We first show \eqref{eq.mainestimatelocal}.  Let
$\omega\subset \tilde\omega\in \PP_{[k-k^{1/4}]}$, with
$k\ge 2n/\eta$, be as in the assertion.  Writing
\[
  \int_{\omega} \biggl( \sum_{j=k}^{k+n-1}\xi_j \biggr)^2 \, dm =
  \sum_{j=k}^{k+n-1} \biggl( \int_{\omega}\xi_j^2 \, dm + 2
  \sum_{\ell=j+1}^{k+n-1} \int_{\omega }\xi_j \xi_\ell \, dm
  \biggr) \,,
\]
it is sufficient to show that
\begin{equation}
  \label{eq.firstterm}
  \sum_{j=k}^{k+n-1} \biggl| 1 - \frac1{|\omega|} \int_{\omega}
  \biggl( \xi_j^2 + 2 \sum_{\ell=j+1}^{k+n-1} \xi_j \xi_\ell \biggr)
  \, dm \biggr| =
  O(1) \, .
\end{equation}
Fix $a_0 \in \omega\cap \Omega_*^\varphi$.  By
Corollary~\ref{c.switch} for $\ell_0=2$, we have, for
$k\le j \le k+n-1$,
\begin{multline*}
  \frac1{|\omega|} \int_{\omega} \biggl( \xi_j^2 + 2
  \sum_{\ell=j+1}^{k+n-1} \xi_j \xi_\ell \biggr) \, dm \\
  = \frac1{|x_v (\omega)|} \int_{x_v (\omega)} \biggl(
  \varphi_{a_0}^2 \circ T_{a_0}^{j-v} + 2 \sum_{\ell=j+1}^{k+n-1}
  \varphi_{a_0} \circ T_{a_0}^{j-v} \varphi_{a_0} \circ
  T_{a_0}^{\ell - v} \biggr) \, dm \\
  + O \bigl( (k+n-j) (\lambda_0^{-(k-k^{1/4})} + |x_v
  (\omega)|^\alpha) \bigr) \,,
\end{multline*}
(recall $v=[k-k^{1/4}]$). Since $0<s<1/q<1$ we have that
  ${1}_{x_v(\omega)} \in H^s_q$, uniformly in $v$ and $\omega$
  (see \cite{Str}), so the first claim of
  Proposition~\ref{p.udc} gives
\begin{align*}
  \int_{x_v(\omega)} (\varphi_{a_0} \circ T_{a_0}^{j-v} &)
  (\varphi_{a_0} \circ T_{a_0}^{\ell-v} )\, dm 
  \\ & = |x_v(\omega)| \int \varphi_{a_0} \cdot
      (\varphi_{a_0} \circ T_{a_0}^{\ell-j}) \, d\mu_{a_0} +
      O(\rho^{j-v}) \, , \quad \forall \ell \geq j\, .
\end{align*}
Hence,
\begin{multline*}
  \frac1{|\omega|} \int_{\omega } \biggl( \xi_j^2 + 2
  \sum_{\ell=j+1}^{k+n-1} \xi_j \xi_\ell \biggr) \, dm = \int
  \biggl( \varphi_{a_0}^2 + 2 \sum_{\ell=k+1}^{k+n-1}
  \varphi_{a_0} \cdot (\varphi_{a_0} \circ T_{a_0}^{\ell-j}
  ) \biggr) \,
  d\mu_{a_0} \\
  + O \bigl( (k+n-j) (\lambda_0^{-(k-k^{1/4})} + |x_v
  (\omega)|^\alpha  + \rho^{j-v} |x_v (\omega)|^{-1})\bigr)
  \,.
\end{multline*}
By \eqref{eq.sigma} and \eqref{eq.varphinormalised}, we have
\[
  1 = \int \varphi_{a_0}^2 \, d \mu_{a_0} + 2 \sum_{i=1}^\infty
  \int \varphi_{a_0} \cdot \varphi_{a_0} \circ T_{a_0}^i \, d
  \mu_{a_0} \,.
\]
Therefore, the second claim of Proposition~\ref{p.udc}
gives
\[
  \int \biggl( \varphi_{a_0}^2 + 2 \sum_{\ell=j+1}^{k+n-1}
  \varphi_{a_0} \cdot (\varphi_{a_0} \circ T_{a_0}^{\ell-j})
  \biggr) \, d\mu_{a_0} = 1 + O (\rho^{k+n-j}) \, .
\]
Hence, we find, for $k \leq j \leq k + n - 1$ and $v = [k -
k^{1/4}]$,
\begin{multline*}
  \biggl| 1 - \frac1{|\omega|} \int_{\omega} \biggl( \xi_j^2 + 2
  \sum_{\ell=j+1}^{k+n-1} \xi_j \xi_\ell \biggr) \biggr| \\ \leq
  C (k+n-j) \bigl (\lambda_0^{-(k-k^{1/4})} + |x_v
  (\omega)|^\alpha + \rho^{j - v} |x_v (\omega)|^{-1} \bigr) + C
  \rho^{k+n-j}\,.
\end{multline*}
To proceed we shall use several times that
\[
  \sup_{n}\sup_k \sum_{j=k}^{k+n-1} \frac{1}{(k+n-j)^2} \le
  \sup_n \sum_{\ell=1}^n \frac 1 {\ell^2} <\infty\, .
\]
Clearly, $\rho^{k+n-j} \leq \frac{C}{(k+n-j)^2}$. For the term
$(k+n-j) \rho^{j-v} |x_v (\omega)|^{-1}$, we use
$|x_v (\omega)| \geq \lambda_0^{-k^{1/4}}$ and the definition
\eqref{eq.lambda0} of $\lambda_0$ to get, since $k\ge 2n/\eta$,
\begin{align*}
  \frac{\rho^{j - v} }{|x_v (\omega)|}
  & \leq \rho^{k^{1/4}} \lambda_0^{k^{1/4}} \leq 
    \lambda_0^{-k^{1/4}} \le \frac C {n^3}\leq
    \frac{C}{(k+n-j)^3} \, , \quad k\le j \le k+n-1\,.
\end{align*}
The term $(k+n-j) \lambda_0^{-(k-k^{1/4})}$ is similar.  Finally,
$|x_v (\omega)| \leq v^{-3/\alpha}$ gives
\[
  \sum_{j=k}^{k+n-1} (k + n - j) |x_v (\omega)|^\alpha \leq
  \sum_{j=k}^{k+n-1} \frac{k+n-j}{n^3}\leq n
  \frac{k+n-k}{n^3} = \frac{1}{n}\, .
\]
This proves \eqref{eq.firstterm}, and hence
\eqref{eq.mainestimatelocal}.

\smallskip

We will next deduce \eqref{eq.mainestimate} and
  \eqref{eq.extraestimate} from \eqref{eq.mainestimatelocal}.
Fix $\kappa_1>\kappa_0$, let $N_1(\kappa_1)\ge N_0$ be given by
Lemma~\ref{l.goodpartition}, and let $K\ge N_1$ be such that
$ k^{\kappa_1}\le \lambda_0^{k^{1/4}}$ for all $k \ge K$.  Then,
if $v = v(k) \ge K$ (so that $k\ge K$), we have
\begin{equation}\label{theproof}
  |x_v (\tilde \omega)| > v^{-\kappa_1} = [k - k^{1/4}]^{-\kappa_1}
  > \lambda_0^{-k^{1/4}}\, , \quad \forall \tilde \omega \in
  \mathcal{P}_v \, .
\end{equation}
Refining
$\mathcal{P}_v$ to a partition $\QQ_v$ such that
\[
  \lambda_0^{-k^{1/4}} \leq |x_v (\omega)| \leq v^{-3/\alpha} \,
  , \qquad \forall \omega \in \QQ_v\, ,
\]
we set  $\Omega_{*,v}^\QQ$ as in \eqref{newdef}
and we decompose
\begin{align*}
  |\Omega^\varphi_*| \cdot E \Biggl(\sum_{j=k}^{k+n-1} \xi_j
  \Biggr)^2
  &= \int_{\Omega^\QQ_{*,v}} \Biggl( \sum_{j=k}^{k+n-1} \xi_j
    \Biggr)^2 \, d m - \int_{\Omega^\QQ_{*,v}\setminus
    \Omega^\varphi_*} \Biggl( \sum_{j=k}^{k+n-1} \xi_j \Biggr)^2
    \, d m \\
  &= \int_{\Omega_v} \Biggl( \sum_{j=k}^{k+n-1} \xi_j \Biggr)^2
    \, d m - \int_{\Omega_v\setminus \Omega^\varphi_*} \Biggl(
    \sum_{j=k}^{k+n-1} \xi_j \Biggr)^2 \, d m
    \, .
\end{align*}
Then, using \eqref{lem:measureinomega},
$\sup_k \sup |\xi_k|<\infty$, $\kappa_0 \geq 3 / d_1$, and
$v(k)\ge k/2\ge n/\eta$, 
\begin{align}
  \nonumber
  0 \leq \int_{\Omega^\QQ_{*,v}\setminus \Omega^\varphi_*} \Biggl
  (\sum_{j=k}^{k+n-1} \xi_j \Biggr)^2\, d m
  & \leq \int_{\Omega_v\setminus
    \Omega^\varphi_*} \Biggl (\sum_{j=k}^{k+n-1} \xi_j \Biggr)^2 \,
    d m \\
  & \leq C n^2 e_v \leq C n^2
    n^{1 - d_1 \kappa_0} \leq C  \, ,
    \label{eq:where4isused}
\end{align}
which shows that
\begin{align*}
  \int_{\Omega_*^\varphi} \Biggl (\sum_{j=k}^{k+n-1} \xi_j
  \Biggr)^2 \, d m
  & = \int_{\Omega^\QQ_{*,v}} \Biggl( \sum_{j=k}^{k+n-1}
    \xi_j \Biggr)^2 \, d m + O(1) \\
  &= \int_{\Omega_v} \Biggl( \sum_{j=k}^{k+n-1} \xi_j \Biggr)^2
    \, d m + O(1) \, .
\end{align*}
By \eqref{eq.mainestimatelocal},
\begin{equation}
  \label{cla}
  \int_{\omega} \Biggl (\sum_{j=k}^{k+n-1} \xi_j  \Biggr)^2 \, d m
  = |\omega|(n + O(1)) \, , \qquad \forall \omega \in \QQ_v^*\, .
\end{equation}
Summing \eqref{cla} over $\omega \in \QQ^*_v$, we get that
\[
  \int_{\Omega^\QQ_{*,v}} \Biggl (\sum_{j=k}^{k+n-1} \xi_j
  \Biggr)^2 \, d m = n + O(1)\, .
\]
Finally, using again \eqref{lem:measureinomega} to see
\[
  0 \leq \frac{|\Omega^\QQ_{*,v}|}{|\Omega^\varphi_*|} - 1
  \leq \frac{|\Omega_v|}{|\Omega^\varphi_*|}-1 = O(e_v) =
  O(e_n)\, ,
\]
we have established \eqref{eq.mainestimate}, and
  also \eqref{eq.extraestimate} in the case
  $\Psi_k=(\sum_{j=k}^{k+n-1} \xi_j )^2$ (note that
  $|\Psi_k|\le C n^2 \le C k^2$).  For more general $\Psi_k$, the
  same argument, using $d_1 \kappa_0 \ge 11/3$ in
  \eqref{eq:where4isused}, gives \eqref{eq.extraestimate}.  This ends the proof of
  Proposition~\ref{p.mainestimate}.

\section{Proof of Theorem~\ref{t.asip} via Skorokhod's
  Representation Theorem}
\label{s.asip}
We will rearrange the Birkhoff sum as a sum of blocks of
polynomial size, approximate the blocks by a martingale, and
finally apply Skorokhod's representation theorem to this
martingale. The size for the $j$th block $\III_j$ is $j^{2/3}$,
which will give the error exponent $\gamma>2/5$ in our
ASIP.\footnote{\label{choiceb2}A block size $\#\III_j=j^b$
  replaces $3/5$ in \eqref{eq.MN} by $1/(1+b)$, so that the first
  constraint becomes $N^\gamma> N^{b/(1+b)}$, see
  \eqref{eq.lastterm}. Our bounds \eqref{eq.43}--\eqref{eq.43'}
  (with G\'al--Koksma and $M(N)\sim N^{1/(1+b)}$) give
  $N^{\gamma}> N^{(b+2)/(4(b+1))}$. Hence, $b=2/3$ is the
  optimum. In the iid case a block size $j^{1/2}$ gives
  $\gamma >1/3$ (\cite[p.~25]{ps}), see also the beginning of
  \cite[Sec. 6]{DS}.}

\subsection{Blocks $\III_M$. Approximations $\chi_i$ and $y_j$}
\label{ss.blocks}

Fix $a_*$, $\varpi\in (0,1)$, $q$, $s\in (0, \min \{\varpi, 1/q\})$, $\rho$, $\theta$, $\lambda_0$,
$\eta$, $\alpha$, $\varphi\in C^\varpi$,
$\Omega_*^\varphi=\Omega_*\cap[a_*-\epsilon_\varphi,
a_*+\epsilon_\varphi]$ as in the beginning of
Section~\ref{s.switch}.
Set 
\[
  \PP_{*,k}:=\{\omega \in \PP_{k} \mid |\omega\cap
  \Omega_*^\varphi|>0\} \, , \qquad \Omega_{*, k}:=
  \bigcup_{\omega \in \PP_{*, k}} \omega, \qquad k \ge 1\, .
\]

Fix $\gamma\in (2/5, 1/2)$ and\footnote{See Lemma~\ref{l.LLNY}
  for the condition $\delta<2(\gamma-2/5)$.}
$\delta \in (0,\min \{1/5, 2(\gamma-2/5)\})$.  For $i\ge 1$, we
shall approach
$\xi_i \colon [a_*-\epsilon_\varphi,a_*+\epsilon_\varphi] \to
\complex$ (see Remark~\ref{Whitney}) by the stepfunction
\[
  \chi_i\colon \Omega_{*, r_i} \to \complex \, , \qquad \chi_i =
  E(\xi_i | \FF_{r_i}) \, ,\,\,\,\, \mbox{ where }
  r_i=i+[i^{\delta}] \, ,
\]
with $\FF_k$ the $\sigma$-algebra generated by the intervals in
$\PP_{*,k}$. Conditional expectations are only
  defined almost everywhere, but we
may set (see \eqref{expect})
\begin{equation}\label{defchi}
  \chi_i|_{\omega} \equiv \frac{ \int_{\omega \cap
      \Omega_*^\varphi} \xi_i \, dm}{|\omega \cap
    \Omega_*^\varphi|} \, , \qquad  \forall \omega \in
  \PP_{*,r_i} \, , \, \forall  i \ge 1\, .
\end{equation}
Thus, $\chi_i$ is defined everywhere on
$\Omega_{*,r_i}$, allowing pointwise claims about it.

Recalling $e_\ell$ from \eqref{lem:measureinomega}  and our
assumption $\lambda_{CE}>e^{ 14 \alpha_{BC}}$ in the proof of Proposition~\ref{keyprop}, we have the
following basic lemma:
\begin{lemma}
  \label{l.chiapprox}
  For any 
$\tilde \lambda_{CE}\in (e^{\alpha_{BC}}, \sqrt {\lambda_{CE} } \cdot e^{-\alpha_{BC}})$, there exists $C$ such that
  \begin{equation}
    \label{eq.chia1}
    |\xi_i(a) - \chi_i(a)| \leq C \tilde \lambda_{CE}^{- 
      \theta i^{\delta}} \, , \qquad \forall i\ge 1\, ,\ \forall
    a \in \Omega_{*, r_i} \, ,
  \end{equation}
  and\footnote{The constant $C$ in \eqref{eq.chia2} goes to
    infinity as $\delta\to 0$, i.e.\ if $\gamma\to 2/5$.}  for
  all $i\ge1$, $j\ge0$ and all
    $a\in \Omega_{*,r_i}$
  \begin{equation}
    \label{eq.chia2}
    |E(\xi_{i+j} | \FF_{r_i} )(a)| = |E(\chi_{i+j} |
    \FF_{r_i})(a)| \leq
    C \min (1, 
    e_{ [\eta (j - 2 i^\delta)]}) \, .
  \end{equation}
\end{lemma}

Following \cite[Sec 3.3]{ps}, \cite[Sec 6.1]{DS}, we define
inductively consecutive blocks $\III_j$ of integers and
associated functions $y_j$: Let $\III_1 = \{1\}$, and let
$\III_j$ for $j\ge 2$ contain $[j^{2/3}]$ consecutive
integers. The first blocks are below:
\[
\underbrace{1,}_{\III_1}
\underbrace{2,}_{\III_2}
\underbrace{3, 4,}_{\III_3}
\underbrace{5, 6,}_{\III_4}
\underbrace{7, 8,}_{\III_5}
\underbrace{9, 10, 11,}_{\III_6}
\underbrace{12, 13, 14,}_{\III_7}
\underbrace{15, 16, 17, 18,}_{\III_8}
\underbrace{19, 20, 21, 22,}_{\III_9}
\ldots
\]

Let $M = M(N)$  be uniquely defined by $N \in \III_M$.
There exists $C$ such that
\begin{equation}
  \label{eq.MN}
  C^{-1} N^{3/5} \leq M(N) \leq C N^{3/5} \, , \qquad \forall N \geq
  1 \, .
\end{equation}
By \eqref{eq.chia1} in Lemma~\ref{l.chiapprox}, there is $C$ such
that, for all $i\ge 1$ and all $a \in \Omega_{*, r_i}$,
\begin{align}
  \label{eq.lastterm}
  \biggl| \sum_{i=1}^N\xi_i(a) - \sum_{j=1}^{M(N)}  \sum_{i \in
  \III_j} \chi_i (a) \biggr|
  & \le\sum_{i=1}^N|\xi_i(a)-\chi_i(a)|+C\# \III_M  \le
    C   N^{2/5}\,  ,
\end{align}
for all $N\ge1$. Hence, in order to prove Theorem~\ref{t.asip}, it is
sufficient to consider 
\[
  y_j \colon \Omega_{*, [Cr_j^{5/3}]} \to \complex\, , \qquad y_j
  := \sum_{i \in \III_j} \chi_i \, , \qquad j \ge 1\, .
\]

\begin{proof}[Proof of Lemma~\ref{l.chiapprox}]
 By \eqref{defchi}, since $\xi_i$ is continuous (see
  Remark~\ref{Whitney}), for any $\omega\in \PP_{*,r_i}$, there
  exists $a'\in \omega$ such that $\chi_i|_{\omega}=\xi_i(a')$.
  Revisiting the decomposition \eqref{decompose}, and using
  \eqref{eq.alphal1infty} and the $\theta$-H\"older continuity of
  $a\mapsto \varphi_a(u)$ (as for \eqref{thetagamma}), we find
  $C$ such that for all $i\ge 1$ and $\omega \in \PP_{*,r_i}$
  \begin{align*}
    |\xi_i(a)-\chi_i(a)|=   |\xi_i (a)
    - \xi_i (a')| 
    & \leq C ( |\omega|^\theta+  |x_i (\omega)|^\varpi  ) \leq C
      |x_i (\omega)|^{\theta }\, , \, \forall a \in \omega\, , 
  \end{align*}
  where we used $\theta \le \varpi$ and \eqref{small} in second
  inequality. This establishes \eqref{eq.chia1}, since for any $\bar \lambda_{CE}\in (e^{\alpha_{BC}}, \sqrt {\lambda_{CE} } \cdot e^{-\alpha_{BC}})$, there exists $\bar C$ such that
\begin{equation}  \label{eq.ri}
  |x_i(\omega)| \leq \bar C \cdot \bar \lambda_{CE}^{-i^\delta} \cdot i^{\kappa_0}
  \,, \qquad \forall 
 \omega \in \PP_{*,r_i} \, , \, \, \forall i \, .
\end{equation}
To show \eqref{eq.ri} first note, using \eqref{eq.distortion3}, that there exists $a\in \omega$ such that
\[
  |x_i(\omega)| \leq C
  \frac{|x_{r_i}(\omega)|}{|(T_a^{i^\delta})'(x_i(a))|} \, .
\]
Then, if $a\in \Omega_*$,  the
polynomial recurrence \eqref{eq:polapproach} and standard arguments give
\begin{equation}\label{err1}
|(T_a^{i^\delta})'(x_i(a))|\ge C i^{-\kappa_0} \bar
\lambda_{CE}^{-i^\delta}
\end{equation} (see e.g.\ \cite[Prop.~3.7]{BS1} in the exponentially recurrent case). If
$a\notin \Omega_*$, we may use bounded distortion
\eqref{eq.distortion1} ($\alpha=0$ suffices here) since
$|\omega\cap \Omega_*|>0$.

  The equality in \eqref{eq.chia2} follows from the definition
  since $\mathcal{F}_{r_i} \subset
    \mathcal{F}_{r_{i+j}}$.  Indeed, for
  $a \in \omega \in \PP_{*,r_i}$,
  \begin{align}
    \label{=4.4}
    |\omega \cap \Omega_*^\varphi|\cdot
    & |E(\xi_{i+j} | \FF_{r_i} )(a)|=
      \int_{\omega \cap \Omega_*^\varphi} \xi_{i+j} \, dm\\
    \nonumber
    & =\sum_{\substack{\omega' \in \PP_{*,r_{i+j}}\\
    \omega'\subset \omega}} 
    |\omega' \cap \Omega_*^\varphi| \cdot
    \frac{ \int_{\omega' \cap \Omega_*^\varphi} \xi_{i+j} \,
    dm}{|\omega' \cap \Omega_*^\varphi|}\\ 
    \nonumber
    &=\sum_{\substack{\omega' \in \PP_{*,r_{i+j}}\\
    \omega'\subset \omega}}|\omega' \cap \Omega_*^\varphi| \cdot
    \chi_{i+j}|_{\omega'} 
    =\sum_{\substack{\omega' \in \PP_{*,r_{i+j}}\\ \omega'\subset
    \omega}} \int_{\omega'\cap \Omega_*^\varphi} \chi_{i+j}\,
    dm\, . 
  \end{align}
  Since $\sup_k \|\xi_k\|_{L^\infty}<\infty$, we may and shall
  assume that $j \geq 2 i^{\delta}$ to prove the upper bound in
  \eqref{eq.chia2}.  For such $j$, recalling $\eta\in (0,1/2)$
  from \eqref{eq.eta}, define
  \begin{equation}\label{fot1}
    k=k(i,j)= \max \biggl\{ i + [i^\delta] + \eta (j - i^\delta),
    \biggl \lceil \frac{i+j}{1 + \eta}\biggr \rceil \biggr\}
  \end{equation}
  so that $k\le i+j-\frac \eta{1+\eta}(j-i^\delta) \le i+j$ and
  $i+j \le k(1+\eta)$.

  Since $\delta$ is fixed, we may and shall assume that $i$ is
 large enough such that $k(i,j)\ge N_1$ (with $N_1$ from
  Lemma~\ref{l.goodpartition}) and
  \begin{equation}\label{largei}
    \max\{\lambda_0^{-(j+i)/(1+\eta)},
\rho^{\eta j/3} \cdot   (2j)^{(\kappa_0+1)/\delta}   \}\le e_{[\eta(j-i^\delta)]} \, . 
  \end{equation}
  Since $k(i,j) \geq r_i$, we have, similarly as for
  \eqref{=4.4},
  \[
    |E(\xi_{i+j} | \FF_{r_i}) (a)| = |E(E(\xi_{i+j} |
    \mathcal{F}_{k(i,j)}) | \FF_{r_i})(a)|\, , \qquad \forall a
    \in \tilde \omega\in \PP_{*, r_i} \, .
  \]
  We must analyse the above decomposition more closely than in
  the proof of \cite[Lemma~6.1]{DS}: Let
  $a \in \tilde \omega\in \PP_{*, r_i}$, then
  \begin{multline}
    |\tilde \omega\cap \Omega_*^\varphi|
    \cdot |E(E(\xi_{i+j} |  \mathcal{F}_{k(i,j)}) |
      \FF_{r_i})(a)| = \biggl | \sum_{\substack{\omega \in
      \PP_{*,k(i,j)}\\
    \omega \subset \tilde \omega}} \frac {|\omega\cap
    \Omega^\varphi_*|} {|\omega\cap \Omega^\varphi_*|}
    \int_{\omega\cap \Omega_*^\varphi} \xi_{i+j} \, d m\biggr |
    \\
    \label{clever}
    \le
      \biggl |\sum_{\substack{\omega \in \PP_{*,k(i,j)}\\\omega
    \subset \tilde \omega}}
    \frac {|\omega|} {|\omega|} \int_{\omega} \xi_{i+j} \, d
    m\biggr | + \sup_{\tilde a} \|\varphi_{\tilde a}\|_{L^\infty}\cdot 
    \sum_{\substack{\omega\in \PP_{*,k(i,j)}\\ \omega \subset
    \tilde \omega}}  |\omega \setminus  (\omega\cap
    \Omega^\varphi_*)| \, .
  \end{multline}

  Since $\tilde \omega\in \PP_{*, r_i}$, the bound
  \eqref{lem:measureinomega1} implies
  \begin{equation}\label{lem:measureinomega2}
    \begin{cases}
      \frac{\sum_{\substack{\omega\in \PP_{*,k(i,j)}\\
            \omega\subset \tilde \omega}} |\omega \setminus
        (\omega\cap \Omega_*)| }{|\tilde \omega\cap
        \Omega_*^\varphi|} \le \frac{d_0 e_{k(i,j)-r_i} | \tilde
        \omega|} {(1-d_0 e_{r_i}) | \tilde\omega|}
      \le C d_0 e_{[\eta(j-i^\delta)]}\, , \\
      |\tilde \omega | /{|\tilde \omega\cap \Omega_*^\varphi|}
      \le \frac{ | \tilde \omega|} {(1-d_0 e_{r_i})|
        \tilde \omega|}\le C \, .
    \end{cases}
  \end{equation}
  In view of
  \eqref{lem:measureinomega1}, \eqref{lem:measureinomega2} and
  \eqref{clever}, it suffices to show
  \[
    \frac1{|\omega|} \biggl| \int_{\omega} \xi_{i+j} \, dm
    \biggr| \le C \min (1, e_{[\eta (j- 2 i^\delta)]}) \, ,
    \qquad \forall \omega \in \PP_{*,k(i,j)}\, .
  \]
  
  Fix $\omega \in \PP_{*,k(i,j)}$. First note that, by
  \eqref{eq.distortion1} for $\alpha=0$,
  \begin{equation}\label{111}
    \frac1{|\omega|} \biggl| \int_\omega \xi_{i+j}(a) \, da
    \biggr| \leq \frac{C}{|x_k(\omega)|} \biggl|
    \int_{x_k(\omega)} \xi_{i+j} (x_k|_\omega^{-1}(y)) \, dy
    \biggr|
    \, .
  \end{equation}
  Then, on the one hand, Lemma~\ref{l.switch} for $\ell_0=1$
  gives $a_0\in \omega\cap \Omega^\varphi_*$ such that
  \begin{align}\label{222}
    \frac1{|x_k(\omega)|} \biggl|
    \int_{x_k(\omega)} ( \xi_{i + j}
    (x_k|_\omega^{-1} (y) )
    & - \varphi_{a_0} (T_{a_0}^{i + j - k}   (y))) \, dy
      \biggr|\\
    & \leq C \lambda_0^{-k(i,j)}
      \nonumber \le  C \lambda_0^{-(i+j)/(1+\eta)} \, .
  \end{align}
  On the other hand, 
  recalling $0<s<1/q$, since $1_{x_k (\omega)}\in H^s_q$
  (uniformly in $k$ and $\omega$), the first claim of
  Proposition~\ref{p.udc}, with
  $\int \varphi_{a_0} d\mu_{a_0}=0$, gives\footnote{A factor
    $|x_k (\omega)|^{-1}$ was omitted when applying
    \cite[Prop.~4.3]{DS} on p.~400 of \cite{DS}: We fix
    this by using our polynomial lower bound on $|x_k(\omega)|$
    (considering two different values of $\delta$ should work for
    \cite{DS}).}
  \begin{align}
    \nonumber  
    \frac1{|x_{k}(\omega)|} \biggl|
    & \int_{x_{k}(\omega)} \varphi_{a_0}(T_{a_0}^{i+j-k}(y))
      \,dy \biggr| \leq C \cdot {k(i,j)^{(\kappa_0+1)}}
      \rho^{i+j-k(i,j)} \\
    \label{333}
    &\, \leq C \cdot (i+j)^{\kappa_0+1}
      \rho^{\eta(j-i^\delta)/(1+\eta)} \leq  C \cdot
      (2j)^{(\kappa_0+1)/\delta} \rho^{\eta
      j/(2+2\eta)} \, . 
  \end{align}
  (We  used $|x_{k}(\omega)| > Ck^{-\kappa_0+1}$ 
  from Lemma~\ref{l.goodpartition}.)  Putting together
  \eqref{111}, \eqref{222}, \eqref{333} and \eqref{largei}, we
  conclude the proof of \eqref{eq.chia2}.
\end{proof}

\subsection{Law of Large Numbers for $y_j^2$}\label{LLN}

Recall that $\gamma\in (2/5,1/2)$ is fixed. The main ingredient
in the proof of Theorem~\ref{t.asip} is the following analogue of
\cite[Lemma~6.2]{DS}, itself inspired by \cite[Lemma~3.3.1]{ps}:

\begin{lemma}
  \label{l.LLNy}
  For  $m_*$-a.e.\
  $a \in \Omega_*^\varphi$,
  there exists  $C(a)$ such that
  \begin{equation}
    \label{eq.yj2}
    \biggl| N - \sum_{j=1}^{M(N)}y_j^2(a) \biggr| \leq
    C(a) N^{2\gamma} \, , \qquad \forall N \geq 1\, .
  \end{equation}
\end{lemma}

The proof of Lemma~\ref{l.LLNy} (which uses
Proposition~\ref{p.mainestimate} and \eqref{eq.chia1}, but not
\eqref{eq.chia2}) is based on the following theorem
(\cite{GalKoksma}, see also \cite[Theorem~A.1]{ps}).

\begin{theorem}[G\'{a}l--Koksma's Strong Law of Large Numbers]
  \label{t.galkoksma}
  Let $z_j$, $j\ge1$, be zero-mean random variables. Assume
  there exist  $p\ge1$ and $C<\infty$ with
  \[
    E \biggl( \sum_{j = m + 1}^{m + n} z_j \biggr)^2 \leq C((m +
    n)^p - m^p) \, , \quad \forall\ m \geq 0\ \text{and}\ n \geq
    1\,.
  \]
  Then for all $\iota>0$, we have
  $\frac1{n^{p/2 + \iota}} \sum_{j = 1}^n z_j \to 0$ almost
  surely.
\end{theorem}

\begin{proof}[Proof of Lemma~\ref{l.LLNy}]
  Set $w_j = \sum_{i \in \III_j} \xi_i$.  Since
  $y_j^2 - w_j^2 = (y_j + w_j) (y_j - w_j)$ and
  $|y_j + w_j|\leq C j^{2/3}$, the bound
  \eqref{eq.chia1} gives $C$ such that
  $|y_j^2 - w_j^2| \leq C j^{2/3}
  \tilde \lambda_{CE}^{-\theta j^\delta }$ for all $j\ge 1$ and
  $a \in \Omega_{*, r_{Cj^{5/3}}}$.  Hence,
  $\sup_{a\in \Omega_*^\varphi}\sum_{j \geq 1} |y_j^2 - w_j^2|$
  is finite, and it suffices to show \eqref{eq.yj2} with $y_j$
  replaced by $w_j$.

  By  \eqref{eq.mainestimate}
  we have
  $| E (w_j^2) - \# \III_j | \leq C$,
  and, since $\sum_{j=1}^{M(N)}   \# \III_j = N$, we get
  $ \bigl| \sum_{j=1}^{M(N)} E( w_j^2 )- N \bigr| \leq C M(N)$.
 Therefore, 
   \begin{equation}
     \label{eq:w2}
     \Biggl| N - \sum_{j=1}^{M(N)} w_j^2 \Biggr| \leq CM(N) + \Biggl|
     \sum_{j=1}^{M(N)} w_j^2 - E( w_j^2 )\Biggr|\,.
  \end{equation}
  
  Assume there exists $C$ such that
  \begin{equation}
    \label{eq.galkoksma}
    E \biggl( \sum_{j = m + 1}^{m + n} w_j^2 - E (w_j^2) \biggr)^2
    \leq C( (m + n)^{8/3 } -m^{8/3})\, , \quad
    \forall m \geq 0,\ n \geq 1\,.
  \end{equation}
  Then Theorem~\ref{t.galkoksma} (G\'al--Koksma) applied to
  $\iota\in(0, 10(\gamma-2/5)/3]$, $p=8/3$, and the zero-mean
  random variables $z_j=w_j^2-E( w_j^2)$, implies that
  \[
    \sum_{j = 1}^{M(N)} w_j^2 - E (w_j^2) = o (M^{\frac{4}{3} +
      \iota})\, , \qquad \mbox{almost surely}\, .
  \]
  Hence, \eqref{eq:w2} gives
  $ \bigl| N - \sum_{j=1}^{M(N)} w_j^2(a) \bigr| \leq C(a)
  N^{4/5 + 3\iota/5}\le C(a)N^{2\gamma} $, almost surely (recall
  $M(N) \sim N^{3/5}$ by \eqref{eq.MN}).
  It  remains to prove \eqref{eq.galkoksma}.
  
  By Jensen's inequality we have $(E(w_j^2))^2 \leq E(w_j^4)$ and
  therefore
  \begin{multline}
    \label{eq.galkoksmapre}
    E \biggl(\sum_{j = m + 1}^{m + n} w_j^2 - E (w_j^2) \biggr)^2
    \\ \leq 2 \sum_{j = m + 1}^{m + n} \biggl( E (w_j^4) +
    \sum_{k = j + 1}^{m + n} |E (w_j^2 w_k^2) - E (w_j^2) E
    (w_k^2)| \biggr) \, .
  \end{multline}
  We consider first $E(w_j^4)$. Fix $\upsilon \in (0,1/6)$ and,
  for $j\ge 1$,
  let
  \begin{multline*}
    S_j = \{\, \vec v \in \III_j^4 \mid v_1 \leq v_2 \leq v_3
    \leq v_4 \mbox{ and } \max\{v_2 - v_1 , v_4-v_3\}\ge
    j^\upsilon \,\}.
  \end{multline*}
  Then, since
  $\#(  \{\, \vec v \in \III_j^4 \mid v_1 \leq v_2 \leq v_3
    \leq v_4\} \setminus S_j)\le (j^{2/3 + \upsilon})^2 =
  j^{4/3+2\upsilon}$, we find
  \begin{align}
    \int\limits_{\Omega_*^\varphi}
    w_j(a)^4 \, da
    &= \sum_{\vec v \in \III_j} \biggl|
      \int_{\Omega_*^\varphi} \prod_{\ell = 1}^4 \xi_{v_\ell} (a)
      \, da \biggr| \nonumber  \leq C \sum_{\substack{\vec v \in
      \III_j^4 \\ v_1 \leq \ldots \leq v_4}} \biggl|
    \int_{\Omega_*^\varphi } \prod_{\ell=1}^4 \xi_{v_\ell} (a) \,
    da \biggr| \nonumber \\
    & \leq C \sum_{\vec v\in S_j} \biggl|
      \int_{\Omega_*^\varphi} \prod_{\ell=1}^4\xi_{v_\ell}(a) \,
      da \biggr| + C j^{4/3+2\upsilon} \, .
    \label{eq.SSS}
  \end{align}
  Let $\vec v\in S_j$ be such that $v_4-v_3 \geq j^\upsilon$.
  For $\omega \in \PP_{v_3}$ such that
  $\omega\cap \Omega_*^\varphi\ne \emptyset$, the change of
  variable in equation \eqref{eq.switchdist}, together with an
  easy variant of Lemma~\ref{l.switch} deduced from
  \eqref{noell0}, give $a_0\in \omega\cap \Omega_*^\varphi$ such
  that
  \begin{multline*}
    \frac{1}{|\omega|} \biggl| \int_\omega \prod_{\ell=1}^4
    \xi_{v_\ell}(a) \, da \biggr| \\ \leq \frac
    C{|x_{v_3}(\omega)|} \biggl| \int_{x_{v_3}(\omega)} \biggl(
    \prod_{\ell=1}^3 \xi_{v_\ell} ( x_{v_3}|_\omega^{-1} (y)) \biggr)
     \varphi_{a_0} (T_{a_0}^{v_4 - v_3} (y) ) \, dy \biggr|
    + C \lambda_0^{-v_3} \, .
  \end{multline*}
  For $y\in x_{v_3}(\omega)$, setting $a = x_{v_3}|_\omega^{-1} (y)$,
  and recalling Remark~\ref{Whitney}, we find
  \begin{align*}
    |\xi_{v_\ell} (x_{v_3}|_\omega^{-1} (y))
    & - \varphi_{a_0} (x_{v_\ell} \circ x_{v_3}|_\omega^{-1}(y))|
    \\ 
    & = |\varphi_a  (x_{v_\ell} \circ x_{v_3}|_\omega^{-1}(y)) -
      \varphi_{a_0} (x_{v_\ell} \circ x_{v_3}|_\omega^{-1}(y))|
      \leq C |\omega|^\theta\, ,
  \end{align*}
  for $\ell=1, 2, 3$. Thus, \eqref{lem:measureinomega} and
  \eqref{small} imply (using $\sup_k\|\xi_k\|_{L^\infty}<\infty$)
  \begin{multline}
    \label{eq.xxxx}
    \biggl| \int_{\Omega^\varphi_*} \prod_{\ell=1}^4 \xi_{v_\ell}
    (a) \, da \biggr| \leq C e_{v_3}+\sum_{\omega \in
      \PP_{*,v_3}} | \omega |\biggl[ \frac
    {C\lambda_{CE}^{-v_3\theta}}{|x_{v_3}(\omega)|} + C
    \lambda_0^{-v_3} \biggr]\\
    +\sum_{\omega \in \PP_{*,v_3}} | \omega | \frac
    C{|x_{v_3}(\omega)|} \Bigl| \int_{x_{v_3}(\omega)} \Bigl(
    \prod_{\ell=1}^3 \varphi_{a_0}(x_{v_\ell} \circ(
    x_{v_3}|_\omega^{-1})(y)) \Bigr)
    \varphi_{a_0}(T_{a_0}^{v_4-v_3}(y)) \, dy \Bigr| \,.
  \end{multline}
We claim that, for $\ell=1,2,3$, and 
for each $\omega\in \PP_{*, v_3}$,
  \begin{equation}\label{err2}
    |\partial_y(x_{v_\ell} \circ( x_{v_3}|_\omega^{-1}))(y)| \leq
    C v_{\ell}^{\kappa_0}\, , \quad \forall y \in
    x_{v_3}(\omega) \, .
  \end{equation}
Indeed, by \eqref{eq.distortion3}, there exists $a \in \omega$ such that
  \[
    |\partial_y (x_{v_\ell} \circ( x_{v_3}|_\omega^{-1}))(y)|\le
    C |(T_a^{v_3-v_\ell})'(T_a^{v_\ell+1}(c))|^{-1} \, .
  \]
  Thus, if $a\in \Omega_*$, standard arguments (see e.g.\
  \cite[Prop.~3.7]{BS1}, using our polynomial recurrence
  \eqref{eq:polapproach}) give the claim.  Otherwise, since
  $|\omega\cap \Omega_*|>0$, we may use \eqref{eq.distortion1} as
  for \eqref{eq.ri}.
 
 Therefore, we find $C$ such that for each $v_3$ and $\omega\in \PP_{*, v_3}$,
  \[
    \| 1_{x_{v_3} (\omega)} \cdot \prod_{\ell=1}^3(\varphi_{a_0} \circ
    x_{v_\ell} \circ x_{v_3}|_\omega^{-1}) \|_{H^s_q}
    \leq C (v_1 v_2)^{\varpi \kappa_0}\| \varphi_{a_0} \|^2_{C^\varpi} \| \varphi_{a_0} \|_{H^s_q} \, .
  \]
  Indeed, on the one hand,  
there exists $C$ such that, for any $C^2$ map $\TT$, we have
\[
\|\varphi_{a_0}\circ \TT \|_{C^\varpi}\le
C \sup|\TT'|^\varpi \|\varphi_{a_0}\|_{C^\varpi}
\, .
\]
On the other hand, since $0<s<1/q<1$, the characteristic function of an interval is
a  bounded multiplier on $H^s_q(I)$ (uniformly in the size of the
  interval), and since $s<\varpi$, a function in $C^\varpi$  is a bounded
multiplier on $H^s_q(I)$ (\cite{Str, Tho}).

  Hence, by the first claim of Proposition~\ref{p.udc}
  (with \eqref{eq.alphal1infty} and
  $\int \varphi_{a_0} d\mu_{a_0}=0$), we have
  \begin{align*}
    \frac{\bigl| \int_{x_{v_3}(\omega)} \bigl( \prod_{\ell=1}^3
      \varphi_{a_0}(x_{v_\ell} \circ x_{v_3}|_\omega^{-1}) \bigr)
      \varphi_{a_0}(T_{a_0}^{v_4-v_3}) \, dy
      \bigr|}{|x_{v_3}(\omega)|} 
&\le C (v_1 v_2)^{\varpi \kappa_0}
    \frac{\rho^{v_4-v_3}}{|x_{v_3}(\omega)|}\\
&\le C j^{10\varpi \kappa_0/3}
    v_3^{\kappa_1}\rho^{j^\upsilon} \, .
  \end{align*}
  (We used Lemma~\ref{l.goodpartition} and that $v_\ell\in \III_j$,
  implies $v_\ell \leq C j^{5/3}$,).  Next,
  \[
    \biggl| \int_{\Omega^\varphi_*} \prod_{\ell=1}^4 \xi_{v_\ell}
    \, da \biggr| \leq C (e_{[Cj^{5/3}]} +
  j^{10\varpi \kappa_0/3}  j^{5 \kappa_1 /3}\rho^{j^\upsilon} +
    j^{5 \kappa_1 /3}\lambda_0^{j^{5/3}}) \leq C e_{[Cj^{5/3}]}\, ,
  \]
  for all $\vec v\in S_j$ with $v_4-v_3\ge j^\upsilon$ (if $j$ is
  large enough).

  \smallskip

  Let now $\vec v\in S_j$ with $v_2-v_1\ge j^\upsilon$. Then
  applying directly Lemma~\ref{l.switch} with $\ell_0=4$, a
  similar reasoning gives  $|\int_{\Omega_*^\varphi}\prod_{\ell=1}^4\xi_{v_\ell}
    \, da|\le C e_{[Cj^{5/3}]}$.

  Finally, since $\# S_j\le \# \III_j^4 \le j^{8/3}$ and
  $e_j\le j^{-d_1 \kappa_0+1}$ with $d_1 \kappa_0\ge 3>9/5$, the
  bound \eqref{eq.SSS} gives $C$ such
  that\footnote{For the purposes of the present
      lemma, a version of \eqref{eq.xi4total} with $Cj^{5/3}$ in
      the right-hand side would suffice.  The stronger statement
      is needed for \eqref{forkappa}.}
  \begin{equation}
    \label{eq.xi4total}
    E(w_j^4) \le C (j^{8/3}e_{[Cj^{5/3}]}+
    j^{4/3 + 2\upsilon})
    \le C j^{4/3+2\upsilon} \, , \qquad\forall j\geq 1\,.
  \end{equation}

  \medskip

  We next bound    $|E (w_j^2 w_k^2) - E (w_j^2) E(w_k^2)|$
  for $k\ge j+1$. If $k = j + 1$, by
  Cauchy's inequality and \eqref{eq.xi4total},
  \[
    E (w_j^2 w_{j+1}^2) \leq \sqrt{ E (w_j^4) E (w_{j+1}^4)} \leq
    C j^{5/3}\, .
  \]
  By \eqref{eq.mainestimate}
we have
  $  E (w_{j}^2) E( w_{j + 1}^2) \leq C j^{4/3}$.
  Hence
  \begin{equation}
    \label{eq:jj+1}
    |E (w_j^2 w_{j + 1}^2) - E (w_j^2) E( w_{j + 1}^2)| \leq C
    j^{5/3 }\, .
  \end{equation}
  Assume now that $k \geq j+2$. By construction, $y_j$ is
  constant on elements of $\PP_v$ if
  $ v \geq r_{j_1}=j_1 + [j_1^\delta] $, where $j_1$ is the
  largest number in $\III_j$.  Let 
  \[
    k_0=k_0(k):=\min \III_{k} \ge \frac{k^{5/3}}{C}\, .
  \]
  Then, for large enough $j$,
  using that $x\mapsto x-x^{1/4}$ is increasing for large $x$, we
  find
  \[
    k_0 - k_0^{1/4} \geq j_1 + \#\III_{j+1} - (j_1 + \#\III_{j+1})^{1/4} \geq j_1 + 2 j_1^{2/3} - 2 j_1^{1/4} \geq
    j_1 + j_1^{1/4}\, .
  \]
  Since $k\ge j+2$ and $\delta < \frac{1}{4}$, we have that $y_j$ is constant on elements of $\PP_v$ for
  \[
    v=v(k_0)=[k_0-k_0^{1/4}]\, .
  \]
   
  Lemma~\ref{l.goodpartition} gives
  $|x_{v (k_0)}(\omega)| \geq \lambda_0^{-{k_0}^{1/4}}$ if
  $\omega \in \PP_{v(k_0)}$.  Thus, there exists a refinement
  $\QQ_{v(k_0)}$ of $\PP_{v(k_0)}$ such that,
  \[
    \lambda_0^{-{k_0}^{1/4}} \leq |x_{v(k_0)}(\omega)| \leq
    v(k_0)^{-3/\alpha}=
      [k_0-k_0^{1/4}]^{-3/\alpha}\, , \quad \forall \omega \in
    \QQ_{v(k_0)} \, .
  \]
  Therefore, for large enough $k$, the local
  bound~\eqref{eq.mainestimatelocal} in
  Proposition~\ref{p.mainestimate} gives
    for all $\omega \in \QQ_{v}$ with non-empty intersection
    with $\Omega_*^\varphi$ that
    \[
      \biggl| \frac{1}{|\omega|} \int_\omega w_k^2 \, dm -\# \III_k
      \biggr| \leq C \, ,
    \]
    since $n=\# \III_k=[k^{2/3}]\le \eta k_0/2$. 
As in
    \eqref{newdef}, we write $\QQ_{*,v}$ for the set of
    $\omega \in \QQ_{v}$ with nonempty intersection with
    $\Omega_*^\varphi$, and $\Omega_{*,v}^\QQ = \cup \QQ_{*,v}$.  
  Thus, using that $y_j$ is constant on
  each $\om \in\QQ_v$ (since $\QQ_{v}$ refines $\PP_{v}$),
  \begin{align*}
    \int_{\Omega_{*,v}^\QQ} (y_j^2 w_k^2) \, d m 
    &= \sum_{\omega \in \mathcal{Q}_{v}^*} |\omega| \cdot
      y_j^2|_\omega \cdot  \frac{1}{|\omega|}
      \int_\omega w_k^2\, dm\\
    & \in \biggl[ \int_{\Omega_{*,v}^\QQ} y_j^2 \,
      dm (\# \III_k - C),
      \int_{\Omega_{*,v}^\QQ} y_j^2 \, dm
      (\# \III_k + C) \biggr] \, .
  \end{align*}
  Recall that $j\le k-2$. Since
    $|y_j^2| \le Cj^{4/3}\le Ck^{4/3}$, we get
    \[
      \frac 1 {m(\Omega_{*,v}^\QQ)}\int_{\Omega_{*,v}^\QQ} y_j^2
      \, d m = E(y_j^2) + O(1)\, ,
    \]
    by \eqref{eq.extraestimate} applied to $\Psi_k=y_j^2$, and
    since $|y_j^2w_k^2| \le  C k^{8/3}$, we have
    \[
      \frac 1
      {m(\Omega_{*,v}^\QQ)}\int_{\Omega_{*,v}^\QQ}
      (y_j^2 w_k^2) \, d m = E(y_j^2 w_k^2) + O(1)\, ,
    \]
    by \eqref{eq.extraestimate} applied to $\Psi_k=(y_j^2
    w_k^2)$.
  That is,
  \[
    |E (y_j^2 w_k^2) - \#\III_k E (y_j^2)| \leq  C
      (E (y_j^2) + 1) \, .
  \]
  Next, the global estimate \eqref{eq.mainestimate} in
  Proposition~\ref{p.mainestimate} gives
  $|E (y_j^2) E (w_k^2) - \# \III_k E (y_j^2)| \leq C E
  (y_j^2)$. Therefore\footnote{The expression
    $\#\III_j=j^b=j^{2/3}$ in the right-hand side already leads
    to $\gamma>2/5$. See Footnote~\ref{choiceb2}.}
  \begin{equation*}
    |E (y_j^2 w_k^2) - E (y_j^2 )E (w_k^2)| \leq 
      C (2E( y_j^2) + 1) 
    )\leq C \# \III_j  \, .
  \end{equation*}
  Hence, for large enough $j$ and all $k\ge j+2$, since
  $\sup|w_j+y_j|\le C \# \III_{j}$,
  \begin{align}
    \nonumber
    & |E(w_j^2w_k^2)-E(w_j^2)E(w_k^2)| \le |E (y_j^2 w_k^2) - E
      (y_j^2 )E (w_k^2)| \\
    \nonumber
    & \qquad \qquad + |E (y_j^2 w_k^2) - E    (w_j^2
      w_k^2)|  +   |E (w_j^2) E (w_k^2) - E( y_j^2) E
      (w_k^2)|  \\
    \nonumber
    & \leq C  \# \III_j + C  E(w_k^2)\sup|w_j - y_j| \cdot \sup
      |w_j + y_j|  \\ 
    \label{eq.43}
    & \leq Cj^{2/3} + C k^{2/3}j^{2/3}
      \tilde \lambda_{CE}^{-\theta j^{5\delta/3}}  \,.
  \end{align}
  (We used \eqref{eq.chia1} to get
  $\sup |w_j - y_j|\le C \# \III_j \tilde \lambda_{CE}^{-\theta
    j^{5\delta/3}}$.)

  Finally, we plug \eqref{eq.43}, \eqref{eq:jj+1},
  \eqref{eq.xi4total} into \eqref{eq.galkoksmapre}, and get,
  since $2\upsilon<1/3$,
  \begin{align}
    \nonumber
    E \biggl(&\sum_{j = m + 1}^{m + n} w_j^2 - E (w_j^2)
               \biggr)^2 \\
    \nonumber
             & \leq C \sum_{k = m + 3}^{m + n}k^{2/3}
               \sum_{j=m+1}^{\infty} j^{2/3}
               \tilde \lambda_{CE}^{-\theta j^{5\delta/3}}+
               C \sum_{j = m + 1}^{m + n} \biggl(
               j^{5/3 } 
               + \sum_{k = j + 2}^{m + n} j^{2/3}\biggr)  \\
             & \leq C   \Biggl ((m+n)^{5/3} -
               m^{5/3}+ \sum_{j = m + 1}^{m + n} \biggl(
               j^{5/3 } + (m+n-j) j^{2/3}  \biggr )    \Biggr )  
               \label{eq.43'} 
               \, . 
  \end{align}
  This proves \eqref{eq.galkoksma}.
\end{proof}

\subsection[Martingale Differences $Y_j$ --- Skorokhod's
  Theorem]{Martingale Differences $Y_j$. Skorokhod's
  Representation Theorem}
\label{last}

As in Schnellmann's adaptation of \cite[Section~3.4--3.5]{ps} in
\cite[Section~6.3]{DS}, let $\LL_j$ be the $\sigma$-algebra
generated by $\{y_\ell\}_{1\le \ell \le j}$, and set
\begin{align}\label{defYj}
  u_j &= \sum_{k \geq 0} E (y_{j + k} \mid \LL_{j - 1}) \, ,
        \qquad Y_j = y_j + u_{j + 1} - u_j \, , \qquad j\ge 2 \,
        . 
\end{align}
Then $\{Y_j,\LL_j\}$ is a martingale difference sequence. Using
\eqref{eq.chia2}, we show that $\{Y_j\}$ inherits the law of
large numbers established for $\{y_j\}$ in Lemma~\ref{l.LLNy}:

\begin{lemma}
  \label{l.LLNY}
  For  $m_*$-a.e.\
  $a \in \Omega_*^\varphi$,
  there exists $C(a)$ such that
  \begin{equation}
    \label{eq.Yj2}
    \Biggl| N -\sum_{j=1}^{M(N)} Y_j^2 (a) \Biggr| \leq
    C(a)N^{2\gamma} \, , \qquad \forall N \geq 1\, , 
  \end{equation}
 and
  \begin{equation}
    \label{eq.Rj2}
    \Biggl| \sum_{j = 1}^{M(N)} E (Y_j^2 \mid \LL_{j - 1}) - Y_j^2 (a)
    \Biggr| \leq C(a) N^{2 \gamma} \, , \qquad \forall N \geq 1\, .
  \end{equation}
\end{lemma}

\begin{proof}
  Recalling the $\sigma$-algebra $\FF_{r_i}$ generated by the
  intervals in $\PP_{r_i}$, we have
  $\LL_{\ell-1} \subset \FF_{r_{i(\ell)}}$, where
  $i(\ell) = \max\{i \in \III_{\ell - 1} \} \leq C \ell^{5/3}$ by
  \eqref{eq.MN}.  Then
  \[
    u_\ell = \sum_{j \geq 1} E ( E(\xi_{i(\ell) + j} \mid
    \FF_{r_{i(\ell)}}) \mid \LL_{\ell - 1}) \, .
  \]
  Since $\sum_{j=1}^\infty e_j < \infty$, the bound
  \eqref{eq.chia2} in Lemma~\ref{l.chiapprox} gives
  \begin{equation}
    \label{eq.uj1}
    |u_\ell (a)| \leq \sum_{j \geq 1} C \min \bigl\{1, e_{ [ \eta
      (j - 2 i(\ell)^\delta))]} \bigr\} \leq \frac {2 C}{
      \eta} i(\ell)^\delta \leq C \ell^{5 \delta / 3} \, .
  \end{equation}
  Put $v_j = u_j - u_{j + 1}$, so that
  $Y_j^2 = y_j^2 -2 y_j v_j + v_j^2$.
  
  We claim that \eqref{eq.Yj2} follows if for a.e.\
  $a\in \Omega^\varphi_*$, there exists $C$ such that
  $\sum_{j = 1}^{M(N)} v_j^2 \leq C N^{4 \gamma-1}$.  Indeed,
  since $\gamma <1/2$, Lemma~\ref{l.LLNy} and Cauchy's inequality
  then give (using $\sum_{j = 1}^{M(N)} y_j^2\le C N$)
  \begin{align*}
    \Biggl| N - \sum_{j = 1}^{M(N)} Y_j^2 \Biggr|
    & = \Biggl| N - \sum_{j = 1}^{M(N)} (y_j^2 -2 y_j v_j + v_j^2)
      \Biggr| \\
    & \leq \Biggl| N - \sum_{j = 1}^{M(N)} y_j^2 \Biggr|+
      \sum_{j=1}^{M(N)}     v_j^2 +
      2 \sqrt{\sum_{j = 1}^{M(N)} y_j^2 \sum_{j = 1}^{M(N)} v_j^2
      } \\
    & \leq C(a) N^{2 \gamma} + C N^{2 \gamma} + C \sqrt{N N^{4
      \gamma-1}} \leq C(a) N^{2 \gamma}\, .
  \end{align*}
  But since  we have
  $ v_j^2 \leq C j^{10\delta/3}$ (by \eqref{eq.uj1}), 
  we find, using $\delta<2(\gamma-2/5)$,
  \[
    \sum_{j = 1}^{M(N)} v_j^2 \leq C M^{1 + 10 \delta / 3} \leq
    N^{3 / 5 + 2 \delta} \leq C N^{4 \gamma-1}
    \, .
  \]
  
  It remains to prove \eqref{eq.Rj2}. Set $R_j = Y_j^2 - E (Y_j^2
  \mid \LL_{j - 1})$ and observe that $\{R_j, \LL_j\}$ is a
  martingale difference sequence. By Minkowski's inequality
  \[
    E (R_j^2) \leq \Bigl(\sqrt{E (Y_j^4)} + \sqrt{E (E (Y_j^2
      \mid \mathcal{L}_{j-1})^2)} \Bigr)^2 \leq \Bigl( 2 \sqrt{E
    (  Y_j^4)} \Bigr)^2 = 4 E (Y_j^4) \, .
  \]
  Since $Y_j = y_j - v_j$, we have, again by Minkowski's
  inequality, 
  \begin{align*}
    E (R_j^2)
    &\leq 4 E (Y_j^4) \leq 4 \Bigl( (E( y_j^4))^\frac{1}{4} +
      (E ( v_j^4))^\frac{1}{4} \Bigr)^4 \leq C (E (y_j^4) + E
      (v_j^4)) \\
    &\leq C (E (w_j^4 )+ E (|w_j^4 - y_j^4|) + E(v_j^4)) \, .
  \end{align*}
  Since
  $w_j^4 - y_j^4 = (w_j^2 + y_j^2) (w_j + y_j) (w_j - y_j)$, we
  get from \eqref{eq.chia1} that $E (|w_j^4 - y_j^4|)$ is
  uniformly bounded.  By \eqref{eq.uj1}, we have
  $|u_j| \leq C j^{5\delta/3}$. Hence,
  $|v_j| \leq |u_j| + |u_{j-1}| \leq C j^{5\delta/3}$, and
  $E (v_j^4) \leq C j^{20 \delta/3} \leq C j^{4/3}$, since
  $\delta<1/5$.  For arbitrary $\iota>0$ the bound
  \eqref{eq.xi4total}, gives $C$ such that
  $E (w_j^4) \leq C j^{4/3 + \iota}$. Thus
  \begin{align}\label{forkappa}
    \sum_{j \geq 1}  \frac{E (R_j^2) }{j^{7/3 + \iota}}< \infty \, ,
  \end{align}
  and a martingale result (see \cite{chow}) implies that
  $\sum_{j\ge1}R_j/j^{7/6+\iota}$ converges almost surely. For $m_*$-a.e.\
  $a \in \Omega_*^\varphi$,
  Kronecker's Lemma gives  $C(a)$ with
  \[
    \sum_{j = 1}^{M(N)} R_j \leq C(a) M^{7/6 + \iota} \leq C(a) N^{21/30 +
      \iota} \, ,
  \]
 using \eqref{eq.MN} in the last inequality. Since
  $21 / 30 <  2 \gamma$ this establishes \eqref{eq.Rj2}.
\end{proof}

We shall apply the following embedding result. 
(See \cite[Theorem~A.1]{hh}.)

\begin{theorem}[Skorokhod's Representation Theorem]
  \label{t.skorokhod}
  For any zero-mean square-integrable martingale
  $\{ \sum_{k=1}^j Y_k,\ \LL_j \mid j \geq1\} $, there exist a
  probability space supporting a (standard) Brownian motion $W$, and
  nonnegative variables $\{T_k\, , \, \, k \geq 1\}$, such that
  $\{ \sum_{k=1}^j Y_k \}_{j \geq 1}$ and
  $\{ W(\sum_{k=1}^j T_k) \}_{j \geq 1}$ have the same
  distribution, and, in addition, letting $\mathcal{G}_0$ be the trivial $\sigma$-algebra (the empty set
   and the entire space), and $\mathcal{G}_j$, for $j\ge 1$, be the
  $\sigma$-algebra generated by
  \[
    \{
W(t) \mid  0 \leq t
    \leq \tau_j\}\, , \,\,\mbox{ where } \tau_j:=\sum_{k=1}^j T_k\,
    ,
 \]
 then $\tau_j$ is $\mathcal{G}_j$-measurable,
 while
 $E (T_1 \mid \mathcal{G}_{0}) = E (W(T_1) ^2 \mid
 \mathcal{G}_{0})$, and
 \[
   E (T_j \mid \mathcal{G}_{j-1}) = E\bigl ( \bigl
   (W(\tau_j)-W(\tau_{j-1}) \bigr)^2 \mid \mathcal{G}_{j-1} \bigr
   )\, , \quad \forall j\ge 2 \, , \quad \mbox{almost surely.}
\]
\end{theorem}

By the last claim of Theorem~\ref{t.skorokhod} and properties of Brownian motion
\begin{equation}
  \label{eq.Tjhelp}
  E (T_j \mid \GG_{j - 1}) = E (W(T_j)^2 \mid \GG_{j - 1}) \, ,
  \qquad \forall j\ge 1\, , 
\end{equation}
almost surely.
(Indeed, letting $W_1$ be an
  independent copy of $W$ we have
  $W(\tau_j) = W_1 (\tau_{j-1} + T_j) = W_1 (\tau_{j-1}) + W
  (T_j)$ in distribution, so that
  $W(\tau_j) - W(\tau_{j-1}) = W(T_j)$ in
  distribution.)

We need one last lemma. Recall that $\gamma\in (2/5,1/2)$ is fixed.

\begin{lemma}[Strong Law of Large Numbers for the Sequence $T_j$]
  \label{l.LLNT}
  For  $m_*$-a.e.\ $a\in \Omega_*^\varphi$, there exists  $C(a)$ such
  that
  \begin{equation}
    \label{eq.Tj}
    \Biggl| N - \sum_{j = 1}^{M(N)} T_j \Biggr| \leq C(a) N^{2 \gamma} \,
    , \qquad \forall N \geq 1 \, .
  \end{equation}
\end{lemma}

\begin{proof} 
To start, apply Theorem~\ref{t.skorokhod} to the 
martingale difference sequence $Y_j$ from \eqref{defYj}, with $\LL_j$
generated by $\{y_\ell\}_{1\le \ell \le j}$.
 Let $\tilde{Y}_j = W(\tau_j) - W(\tau_{j-1})$,
    so that $W (\tau_j) = \sum_{k=1}^j \tilde{Y}_{k}$  and
    $\tilde{Y}_j = W(T_j)$.   By \eqref{eq.Tjhelp}, we have, almost
  surely,
  \begin{align*}
    N - \sum_{j = 1}^{M(N)} T_j = \bigl [N - \sum_{j = 1}^M
    \tilde{Y}_j^2 \bigr] &
+ \sum_{j = 1}^M \bigl[\tilde{Y_j}^2 - E(\tilde{Y}_j^2 \mid\mathcal{G}_{j-1})\bigr] \\
   & + \sum_{j = 1}^M \bigl[ E(T_j \mid \mathcal{G}_{j-1}) - T_j  \bigr] \, ,
\,\,\forall N\ge 1 \, .
  \end{align*}
  Then, since $Y_j$ and $\tilde{Y}_j$ have the
    same distribution, the bound \eqref{eq.Yj2} in Lemma~\ref{l.LLNY} gives
  $C(a)$ such that, for all
  $N\ge 1$, the first sum in the
  right-hand side above is not larger than $C(a) N^{2\gamma}$ . 
  
  For the second sum in the right-hand side
    above, we use \eqref{eq.Rj2}. Since conditional expectations can be expressed in terms of distributions, \eqref{eq.Rj2} is also valid with $Y_j$ replaced by $\tilde{Y}_j$. Thus the second sum in the right-hand side is also bounded by
    $C(a) N^{2\gamma}$ for all $N \geq 1$.  
  
  Finally, let $R_j = E (T_j \mid \GG_{j-1}) - T_j$.  Then
  $\{R_j, \GG_j\}$ is a martingale difference sequence by
  \eqref{eq.Tjhelp}. As in the proof of \eqref{eq.Rj2}, we can estimate
  $
    E( R_j^2) \leq 4 E (W(T_j)^4)$, and thus there exists $C(a)$ such that,
for all $N\ge 1$, we have 
    $\sum_{j=1}^{M(N)} R_j \leq C N^{21/30 + \iota} \leq C(a) N^{2
      \gamma}$ almost surely.
\end{proof}

\begin{proof}[Proof of Theorem~\ref{t.asip}]
  Just like Schnellmann, we follow the proof of
  \cite[Lemma~3.5.3]{ps}, replacing their $1/2-\alpha/2+\gamma$
  by $\gamma$, and replacing Lemma~3.5.1 there by our
  Lemma~\ref{l.LLNT}.  We then obtain that, almost surely,
  \[
    \Biggl| \sum_{j = 1}^{M(N)} Y_j - W (N) \Biggr| = O
    (N^\gamma) \, .
  \]
  Then, using \eqref{eq.uj1} and \eqref{eq.MN}, we find
  \begin{equation}
    \label{eq.yjYj}
    \Biggl| \sum_{j=1}^{M(N)} y_j - Y_j \Biggr|
    = \Biggl| \sum_{j=1}^{M(N)} (u_{j+1} - u_j) \Biggr|
    = |u_{M(N)+1} - u_1| \le C
    N^{\delta}\, .
  \end{equation}
  Since $\delta <2/5$, and recalling \eqref{eq.lastterm}, this
  establishes Theorem~\ref{t.asip}.
\end{proof}

\end{document}